\documentclass[a4paper,10pt]{amsart}
\usepackage[utf8]{inputenc}

\usepackage[
  margin=30mm,
  marginparwidth=25mm,     
  marginparsep=2mm,       
  bottom=25mm,
  ]{geometry}

\usepackage[bbgreekl]{mathbbol}
\usepackage{amsfonts}
\usepackage{latexsym,amssymb,amsthm,mathrsfs,amsmath,amscd,enumerate,enumitem, amscd,color}
\usepackage{mathrsfs}
\usepackage{tikz}
\usepackage{hyperref}
\usetikzlibrary{plotmarks}
\usepackage{soul}
\usepackage{tensor}

\makeatletter
\@namedef{subjclassname@2020}{%
  \textup{2020} Mathematics Subject Classification}
\makeatother

\DeclareSymbolFontAlphabet{\mathbb}{AMSb}
\DeclareSymbolFontAlphabet{\mathbbl}{bbold}
\DeclareSymbolFontAlphabet{\mathbb}{AMSb}
\DeclareSymbolFontAlphabet{\mathbbl}{bbold}
\DeclareFontEncoding{FMS}{}{}
\DeclareFontSubstitution{FMS}{futm}{m}{n}
\DeclareFontEncoding{FMX}{}{}
\DeclareFontSubstitution{FMX}{futm}{m}{n}
\DeclareSymbolFont{fouriersymbols}{FMS}{futm}{m}{n}
\DeclareSymbolFont{fourierlargesymbols}{FMX}{futm}{m}{n}
\DeclareMathDelimiter{\VERT}{\mathord}{fouriersymbols}{152}{fourierlargesymbols}{147}

\newtheorem{thm}{Theorem}[section]
 \newtheorem{cor}[thm]{Corollary}
 
 \newtheorem{prop}[thm]{Proposition}
\theoremstyle{definition}

 \theoremstyle{remark}

\usepackage{hyperref}

\newcommand{\supp}{\mathop{\mathrm{supp}}}

\newcommand{\esssup}{\mathop{\mathrm{ess\, sup \;}}}

\numberwithin{equation}{section}
\allowdisplaybreaks

\begin{document}


\title[]{Quantitative weighted estimates for Harmonic analysis operators in the Bessel setting by using sparse domination.}

\author[V. Almeida]{V\'{\i}ctor Almeida}

\author[J.J. Betancor]{Jorge J. Betancor}

\author[J.C. Fari\~na]{Juan C. Fari\~na}

\author[L. Rodr\'{\i}guez-Mesa]{Lourdes Rodr\'{\i}guez-Mesa}

\address{V\'{\i}ctor Almeida, Jorge J. Betancor, Juan C. Fari\~na, Lourdes Rodr\'{\i}guez-Mesa\newline
	Departamento de An\'alisis Matem\'atico, Universidad de La Laguna,\newline
	Campus de Anchieta, Avda. Astrof\'isico S\'anchez, s/n,\newline
	38721 La Laguna (Sta. Cruz de Tenerife), Spain}
\email{valmeida@ull.edu.es, jbetanco@ull.es,
jcfarina@ull.es, lrguez@ull.edu.es
}

\thanks{The authors are partially supported by grant PID2019-106093GB-I00 from the Spanish Government}



\date{}

\begin{abstract}
In this paper we obtain quantitative weighted $L^p$-inequalities for some operators involving Bessel convolutions. We consider maximal operators, Littlewood-Paley functions and variational operators. We obtain $L^p(w)$-operator norms in terms of the $A_p$-characteristic of the weight $w$. In order to do this we show that the operators under consideration are dominated by a suitable family of sparse operators in the space of homogeneous type $((0,\infty),|\cdot|,x^{2\lambda}dx)$. 
\end{abstract}

\maketitle

\section{Introduction}
In this paper we obtain quantitative weighted $L^p$-inequalities for some operators involving convolutions products in the Bessel setting. We obtain the  $L^p(w)$-operator norm estimates in terms of the $A_p$-characteristic $[w]_{A_p}$ of the weight $w$. In order to do this we prove that the operators under consideration can be dominated pointwisely by a suitable family of sparse operators.

We now define the operators we are going to deal with. Let $\lambda >0$. We denote by $m_\lambda$ the measure on $(0,\infty)$ defined by $dm_\lambda=x^{2\lambda}dx$. $m_\lambda$ is doubling with respect to the usual metric in $(0,\infty)$ defined by the absolute value $|\cdot|$ (see for instance \cite[(1.5)]{YY}). In other words the triple $((0,\infty),|\cdot|,m_\lambda)$ is a space of homogeneous type in the sense of Coifman and Weiss (\cite{CW}).

For every $f,g\in L^1((0,\infty),m_\lambda)$ we define the convolution product $f\#_\lambda g$ of $f$ and $g$ by

$$(f\#_\lambda g)(x)=\int_0^\infty f(y)\; _\lambda\tau_x(g)(y) y^{2\lambda}dy,\;\;\;\;x\in (0,\infty),$$
where, for every $x\in (0,\infty)$,

$$_\lambda\tau_x(g)(y)=\frac{\Gamma(\lambda+1/2)}{\Gamma(\lambda)\sqrt{\pi}}\int_0^\pi g(\sqrt{x^2+y^2-2xy\cos\theta}) \sin^{2\lambda -1}\theta d\theta,\;\;\;y\in (0,\infty).$$
The main properties of $\#_\lambda$ and $_\lambda\tau_x$, $x\in (0,\infty)$, can be found in \cite{Ha} and \cite{Hi}. We define the Hankel transform $h_\lambda f$ of $f\in L^1((0,\infty),m_\lambda)$ as follows

$$h_\lambda(f)(x)=\int_0^\infty (xy)^{-\lambda +1/2}J_{\lambda -1/2}(xy)f(y)y^{2\lambda}dy,\;\;\;\;x\in (0,\infty),$$
where $J_\nu$ represents the Bessel function of first kind and order $\nu>-1$. $h_\lambda$ plays in the harmonic analysis in the Bessel setting the same role than the one of the Fourier transform in the Euclidean harmonic analysis. The operations $\#_\lambda$ and $_\lambda\tau_x$, $x\in (0,\infty)$, are usually named Hankel convolution and Hankel translation because the following equalities hold for every $f,g\in L^1((0,\infty),m_\lambda)$ (\cite{Ha} and \cite{Hi})
$$
h_\lambda(f\#_\lambda g)=h_\lambda(f)h_\lambda(g),
$$
and
$$
h_\lambda(_\lambda\tau_x(g))(y)=2^{\lambda-1/2}\Gamma(\lambda +\frac{1}{2})(xy)^{-\lambda +1/2}J_{\lambda -1/2}(xy)h_\lambda(g)(y),\quad x,y\in (0,\infty).
$$
We now consider the Bessel operator $\Delta_\lambda =-\frac{d^2}{dx^2}-\frac{2\lambda}{x}\frac{d}{dx}$ on $(0,\infty)$. $h_\lambda$, $\#_\lambda$ and $_\lambda\tau_x$, $x\in (0,\infty)$, are connected with $\Delta_\lambda$. For every $f,g\in S(0,\infty)$, the Schwartz space on $(0,\infty)$, we have that
$$
h_\lambda(\Delta_\lambda f)(x)=x^2h_\lambda(f)(x),\quad x\in (0,\infty),
$$
$$
\Delta_\lambda(f\#_\lambda g)=\Delta_\lambda(f)\#_\lambda g,
$$
and, for every $x\in (0,\infty)$,
$$
\Delta_\lambda(_\lambda\tau_x(g))=\; _\lambda\tau_x(\Delta_\lambda g).
$$
The heat and Poisson semigroups associated with $\Delta_\lambda$ are $\#_\lambda$-convolution semigroups. Indeed, if $\{e^{-t\Delta_\lambda}\}_{t>0}$ represents the semigroup generated by $-\Delta_\lambda$ we have that
$$
e^{-t\Delta_\lambda}(f)=W_{\sqrt{2t}}^\lambda\#_\lambda f,\quad t>0,
$$
where
$$
W^\lambda(x)=\frac{2^{1/2-\lambda}}{\Gamma(\lambda +1/2)}e^{-x^2/2},\quad x\in (0,\infty).
$$
The Poisson semigroup $\{e^{-t\sqrt{\Delta_\lambda}}\}_{t>0}$ associated with $\Delta_\lambda$ is defined by
$$
e^{-t\sqrt{\Delta_\lambda}}(f)=P_t^\lambda\#_\lambda f,\quad t>0,
$$
being
$$
P^\lambda(x)=\frac{2\lambda\Gamma(\lambda)}{\Gamma(\lambda +1/2)\sqrt{\pi}}\frac{1}{(1+x^2)^{\lambda +1}},\quad x\in (0,\infty).
$$
If $\phi$ is a complex function defined in $(0,\infty)$ the $t$-dilation $\phi_t$ of $\phi$ is defined in this setting, for every $t>0$, by
$$
\phi_t(x)=\frac{1}{t^{2\lambda +1}}\phi\left(\frac{x}{t}\right),\quad x\in (0,\infty).
$$

The study of harmonic analysis associated with Bessel operators was begun by Muckenhoupt and Stein (\cite{MS}) and continued by Andersen and Kerman (\cite{AK}) and Stempak (\cite{Stem}). In the last fifteen years many problems concerning to the harmonic analysis in the Bessel context has been studied (\cite{BCN}, \cite{BDFS}, \cite{BDT}, \cite{BFS}, \cite{BHNV}, \cite{CZ}, \cite{DLMWY}, \cite{D}, \cite{GS}, \cite{KP}, \cite{K}, \cite{LN}, \cite{NS}, \cite{Vi}, \cite{WYZ} and \cite{YY}).

We now define the operators that we are going to study. As in \cite[Definition 1.4]{YY} we consider the space $Z^\lambda$ that consists of all those $\phi\in C^2(0,\infty)$ satisfying the following properties.

\begin{enumerate}
    \item[(i)] $|\phi(x)|\leq C(1+x^2)^{-\lambda -1},\;\;\;x\in (0,\infty)$, \newline
    \item[(ii)] $|\phi'(x)|\leq Cx(1+x^2)^{-\lambda -2}, \;\;\;x\in (0,\infty)$,\newline
    \item[(iii)]  $|\phi''(x)|\leq C(1+x^2)^{-\lambda -2}, \;\;\;x\in (0,\infty)$.\newline
\end{enumerate}
Let $\phi\in Z^\lambda$. We defined the maximal operator $\phi_*^\lambda$ by
$$
\phi_*^\lambda(f)=\sup_{t>0}|f\#_\lambda\phi_t|,
$$
and the square function $g_\phi^\lambda$ as follows
$$
g_\phi^\lambda(f)(x)=\left(\int_0^\infty|f\#_\lambda\phi_t(x)|^2\frac{dt}{t}\right)^{1/2},\quad x\in (0,\infty).
$$
Let $\rho >2$. The $\rho$-variation operator $V_\rho^\lambda(\{\phi_t\}_{t>0})$ associated with $\{\phi_t\}_{t>0}$ is defined, for every $x\in (0,\infty)$, by
$$V_\rho^\lambda(\{\phi_t\}_{t>0})(f)(x)=\sup_{\substack{0<t_n<t_{n-1}<\cdots<t_1 \\ n\in\mathbb N}}\left(\sum_{j=1}^{n-1}|f\#_\lambda\phi_{t_j}(x)-f\#_\lambda\phi_{t_{j+1}}(x)|^\rho\right)^{1/\rho}.$$
Maximal and  $\rho$-variation operators and square functions for families of operators have had a considerable amount of attention in probability, ergodic theory, and harmonic analysis. As it is well-known the maximal operators $\phi_*^\lambda$ are related with the convergence of $\{f\#_\lambda\phi_t\}_{t>0}$, for instance, as $t\rightarrow 0^+$. If $\phi_*^\lambda$ is bounded in $L^p((0,\infty),m_\lambda)$ for some $1\leq p<\infty$, the almost everywhere convergence of $\{f\#_\lambda\phi_t\}_{t>0}$ as $t\rightarrow 0^+$ is got, for every $f\in L^p((0,\infty),m_\lambda)$ when the almost everywhere convergence is known for every $f\in\mathcal A$ being $\mathcal A$ a dense subset of $L^p((0,\infty),m_\lambda)$. Also, the $L^p((0,\infty),m_\lambda)$-boundedness of the $\rho$-variation operator $V_\rho^\lambda(\{\phi_t\}_{t>0})$ implies almost everywhere convergence of $\{f\#_\lambda\phi_t\}_{t>0}$ as $t\rightarrow 0^+$, for every $f\in L^p((0,\infty),m_\lambda)$ but in this case it is not necessary to know the almost everywhere convergence property for $f$ in some dense subset in $L^p((0,\infty),m_\lambda)$. Furthermore the $\rho$-variation operator $V_\rho^\lambda(\{\phi_t\}_{t>0})$ gives extra information about the convergence of $\{f\#_\lambda\phi_t\}_{t>0}$ as $t\rightarrow 0^+$. For instance $L^p$-boundedness properties of $V_\rho^\lambda(\{\phi_t\}_{t>0})$ allow us to obtain estimates of the $\alpha$-jump that $\{f\#_\lambda\phi_t\}_{t>0}$ has as $t\rightarrow 0^+$. $L^p$-boundedness properties for the square function $g_\phi^\lambda$ lead to $L^p$-estimates for  spectral multipliers for Bessel operators.

Stein, in his monograph \cite{Stein}, established a general theory of harmonic analysis associated with semigroups of operators  that has been fundamental in the development of this area. We refer the interested reader to \cite{CJRW1}, \cite{CJRW2}, \cite{JKRW}, \cite{JSW}, \cite{LeMX1}, \cite{LeMX2}, \cite{Le} and \cite{MTX}, and references therein.

Suppose that $w,b\in L^1_{\rm loc}((0,\infty),m_\lambda)$ and $w_\lambda(I)=\int_I w(x)x^{2\lambda}dx>0$, for every interval $I\subset (0,\infty)$. We say thet $b\in BMO((0,\infty),w,m_\lambda)$ when
$$\|b\|_{BMO((0,\infty),w,\lambda)}:=\sup_{\substack{I\subset(0,\infty) \\ I\;\mbox{interval} }}\frac{1}{w_\lambda(I)}\int_I|b(y)-b_I|y^{2\lambda}dy<\infty.$$
Here, for every interval $I\subset(0,\infty)$, $b_I=\frac{1}{m_\lambda(I)}\int_I b(y)y^{2\lambda}dy$.

The Riesz transform associated with $\Delta_\lambda$ is defined by
$$
\mathcal R_{\Delta_\lambda}=\partial_x\Delta_\lambda^{-1/2}.
$$
In \cite[Theorem 1.3]{DLWY} it was proved that a function $b\in \bigcup_{q>1}L^q_{\rm loc}((0,\infty),m_\lambda)$ is in $ BMO((0,\infty),1,m_\lambda)$ if and only if the commutator $[\mathcal R_{\Delta_\lambda},b]$ defined by
$$
[\mathcal R_{\Delta_\lambda},b]f=\mathcal R_{\Delta_\lambda}(bf)-bR_{\Delta_\lambda}(f)
$$
is bounded from $L^p((0,\infty),m_\lambda)$ into itself, for every $1<p<\infty$.

Let $b$ be a measurable function on $(0,\infty)$. We consider the following higher order commutator operators defined, for every $m\in\mathbb N\setminus\{0\}$, as follows 
$$
\phi_{*,b}^{\lambda,m}(f)(x)=\sup_{t>0}\left|\left(\phi_t\#_\lambda\left[(b(\cdot)-b(x))^m f\right]\right)(x)\right|,\quad x\in (0,\infty),
$$
$$
g_{\phi,b}^{\lambda,m}(f)=\left(\int_0^\infty \left|\left(\phi_t\#_\lambda\left[(b(\cdot)-b(x))^m f\right]\right)(x)\right|^2\frac{dt}{t}\right)^{1/2},\quad x\in (0,\infty),$$
and
\begin{align*} 
  V_{\rho,b}^{\lambda,m}(f)(x) & =
\sup_{\substack{0<t_n<t_{n-1}<...<t_1 \\ n\in\mathbb N}}\Big(\sum_{j=1}^{n-1}\left|\left(\phi_{t_j}\#_\lambda\left[(b(\cdot)-b(x))^m f\right]\right)(x) \right.\\
& -\left. \left(\phi_{t_{j+1}}\#_\lambda\left[(b(\cdot)-b(x))^m f\right]\right)(x)\right|^\rho\Big)^{1/\rho},\;\;\;\;x\in (0,\infty).
\end{align*}
Commutators associated with Littlewood-Paley operators in the Euclidean context were studied in \cite{ChD}, \cite{DX} and \cite{La} and with Hardy-Littlewood maximal operator on spaces of homogeneous type were investigated in \cite{FPW} and \cite{GVW}.

In this paper we obtain quantitative weighted $L^p$-inequalities for the operators $\phi_{*}^{\lambda}$, $\phi_{*,b}^{\lambda,m}$, $g_{\phi}^{\lambda}$, $g_{\phi,b}^{\lambda,m}$, $V_{\rho}^{\lambda}$ and $V_{\rho,b}^{\lambda,m}$.

We say that a measurable function $w$ in $(0,\infty)$ is a weight when $w(x)>0$, for almost all $x\in (0,\infty)$. A weight $w$ in $(0,\infty)$ is said to be in the class $A_p^\lambda(0,\infty)$ provided that

\begin{enumerate}
    \item[(i)] 
    $$[w]_{A_p^\lambda}:=\sup_{\substack{I\subset (0,\infty) \\ I\;\mbox{interval}}}\left(\frac{1}{m_\lambda(I)}\int_I w(x)x^{2\lambda}dx\right)\left(\frac{1}{m_\lambda(I)}\int_I w(x)^{-\frac{1}{p-1}}x^{2\lambda}dx\right)^{p-1}<\infty,$$
    when $1<p<\infty$;
    \item[(ii)]
    $$[w]_{A_1^\lambda}:=\sup_{\substack{I\subset (0,\infty) \\ I\;\mbox{interval}}}\frac{1}{m_\lambda(I)}\int_I w(x)x^{2\lambda}dx\;\esssup_{x\in I}\frac{1}{w(x)}<\infty,$$
    when $p=1$.
\end{enumerate}
By $L^p((0,\infty),w,m_\lambda)$ we understand the $L^p$-space on $(0,\infty)$ with respect to $w(x)x^{2\lambda}dx$.

We are going to state our results

\begin{thm}\label{Th1.1}
Let $1<p<\infty$, $\lambda>0$, $\rho >2$ and $\phi\in Z^\lambda$. Assume that $w\in A_p^\lambda(0,\infty)$. If $T$ represents to $\phi_{*}^{\lambda}$, $g_{\phi}^{\lambda}$ or $V_{\rho}^{\lambda}$, then $T$ can be extended to a bounded operator from $L^p((0,\infty),w,m_\lambda)$ into itself and there exists $C>0$ that does not depend on $w$ such that
$$\|T(f)\|_{L^p((0,\infty),w,m_\lambda)}\leq C[w]_{A_p^\lambda}^{\max\{\frac{1}{p-1},1\}}\|f\|_{L^p((0,\infty),w,m_\lambda)},$$
for every $f\in L^p((0,\infty),w,m_\lambda)$.
\end{thm}

\begin{thm}\label{Th1.2}
Let $1<p<\infty$, $\lambda>0$, $\rho >2$, $m\in\mathbb N\setminus\{0\}$ and $\phi\in Z^\lambda$. Suppose  that $w_1,w_2\in A_p^\lambda(0,\infty)$. If $T$ represents to $\phi_{*,b}^{\lambda,m}$, $g_{\phi,b}^{\lambda,m}$ or $V_{\rho,b}^{\lambda,m}$, then $T$ can be extended to a bounded operator from $L^p((0,\infty),w_1,m_\lambda)$ into $L^p((0,\infty),w_2,m_\lambda)$ and there exists $C>0$ that does not depend on $w_1$ neither $w_2$ such that
$$
\|T(f)\|_{L^p((0,\infty),w_2,m_\lambda)}\leq C\|b\|_{BMO((0,\infty),w,m_\lambda)}^m\left([w_1]_{A_p^\lambda}[w_2]_{A_p^\lambda}\right)^{\frac{m+1}{2}\max\{\frac{1}{p-1},1\}}\|f\|_{L^p((0,\infty),w_1,m_\lambda)},
$$
for every $f\in L^p((0,\infty),w_1,m_\lambda)$, provided that $b\in BMO((0,\infty),w,m_\lambda)$ where $w=\left(\frac{w_1}{w_2}\right)^{\frac{1}{mp}}$.
\end{thm}

Suppose that $\frak m$ is a bounded measurable function on $(0,\infty)$. Then, $\frak m$ defines a $\Delta_\lambda$-spectral multiplier, $T_{\frak m}$ that is bounded from  $L^2((0,\infty),m_\lambda)$ into itself, by
$$T_{\frak m}(\Delta_\lambda)=\int_0^\infty \frak m(s)dE_\lambda(s),$$
where $E_\lambda$ denotes the spectral measure of $\Delta_\lambda$ in $L^2((0,\infty),m_\lambda)$. $T_{\frak m}$ can be written by using the Hankel transformation as follows
$$
T_{\frak m}(\Delta_\lambda)(f)=h_\lambda(\frak m (y^2)h_\lambda(f)),\quad f\in L^2((0,\infty),m_\lambda).
$$
In \cite{DPW} Dziuba\'nski, Preisner and Wr\'obel gave sufficient conditions on the function $\frak m$ that guarantee that the operator $T_{\frak m}(\Delta_\lambda)$ can be extended from $L^2((0,\infty),m_\lambda)\cap L^p((0,\infty),m_\lambda)$ to $L^p((0,\infty),m_\lambda)$ as a bounded operator from $L^p((0,\infty),m_\lambda)$ into itself when $1<p<\infty$ and from $L^1((0,\infty),m_\lambda)$ into $L^{1,\infty}((0,\infty),m_\lambda)$ when $p=1$, provided that $\lambda>1/2$.

The operator $\Delta_\lambda$ satisfies the conditions (A.1) and (A.2) in \cite{BD} (also in \cite{CSZ}) (see \cite[p. 7269]{BDL}). As in \cite[p. 880]{BD} and in \cite[p. 5]{CSZ} we consider a real valued and even function $\psi$ in the Schwartz space $S(\mathbb R)$  such that $\int_0^\infty\psi^2(s)\frac{ds}{s}<\infty$. We define the square function $G_{\psi,\lambda}$ as follows
$$
G_{\psi,\lambda}(f)(x)=\left(\int_0^\infty |T_{\psi_{(t)}}(\Delta_\lambda)f(x)|^2\frac{dt}{t}\right)^{1/2},\quad x\in (0,\infty).
$$
Here $\psi_{(t)}(x)=\psi(t\sqrt{x})$, $t,x\in (0,\infty)$. We can write
$$
G_{\psi,\lambda}(f)(x)=\frac{1}{\sqrt{2}}\left(\int_0^\infty |T_{\psi_{(\sqrt{t})}}(\Delta_\lambda)f(x)|^2\frac{dt}{t}\right)^{1/2},\quad x\in (0,\infty).
$$
We have that
\begin{align*} 
 T_{\psi_{(\sqrt{t})}}(\Delta_\lambda)(f) & =
h_\lambda(\psi(ty)h_\lambda(f)(y)) \\
& =h_\lambda(\psi(ty))\#_\lambda f \\
& =\phi_t\#_\lambda f,\quad f\in L^2((0,\infty),m_\lambda),
\end{align*}
where $\phi=h_\lambda(\psi)$. Since $\psi\in S(\mathbb R)$ is even, we have that $\phi=\widetilde{\phi}_{|(0,\infty)}$ where $\widetilde{\phi}$ is also even and belongs to $S(\mathbb R)$ (\cite[Satz 5 and p. 201]{Al}). Then, $\phi\in Z^\lambda$. Furthermore, we get
$$
G_{\psi,\lambda}=g_{\phi}^{\lambda}.
$$
Our result in Theorem \ref{Th1.1} for the square function extends the ones in \cite[Theorem 1.6]{BD} and in \cite[Theorem 2.3]{CSZ} for Bessel operators $\Delta_\lambda$. The method we use to prove Theorem \ref{Th1.1} is different than the one in the mentioned papers. Furthermore, we note that in our estimate the exponent of $[w]_{A_p}$ is $\max\{1,\frac{1}{p-1}\}$ while in \cite[Theorem 1.6]{BD} and \cite[Theorem 2.3]{CSZ} the corresponding exponent is $\max\{\frac{1}{2},\frac{1}{p-1}\}$.

An important spectral multiplier is that one named Bochner-Riesz multiplier. The $\alpha$-Bochner-Riesz multiplier associated with $\Delta_\lambda$ is defined by
$$
B_{\lambda,t}^\alpha(f)=h_\lambda\left(\left(1-\left(\frac{y}{t}\right)^2\right)_+^\alpha h_\lambda(f)\right),\quad t>0,$$
where $z_+=\max\{0,z\}$, $z\in\mathbb R$, and $\alpha >0$. Since (\cite[(7) p. 48]{EMOT})
$$
h_\lambda\left(2^\alpha\Gamma(\alpha +1)y^{-\alpha -\lambda -1/2}J_{\alpha+\lambda+1/2}(y)\right)(x)=(1-x^2)_+^\alpha,\quad x\in (0,\infty),$$
we have that
$$
B_{\lambda,t}^\alpha(f)=\phi_{1/t}^{\lambda ,\alpha}\#_\lambda f,\quad t>0,
$$
where $\phi ^{\lambda ,\alpha}(x)=2^\alpha\Gamma(\alpha +1)x^{-\alpha -\lambda -1/2}J_{\alpha+\lambda+1/2}(x)$, $x\in (0,\infty)$.

The Bochner-Riesz multiplier in the Bessel setting was studied in \cite{BC} and \cite{CTV}. 

According to \cite[(5.4.3) and (5.11.6)]{Leb} it follows that the function $x^{-\nu}J_\nu(x)$ is bounded in $(0,1)$ and the function $\sqrt{x}J_\nu(x)$ is bounded in $(1,\infty)$ when $\nu>-1/2$. Then,
$$
|\phi^{\lambda,\alpha}(x)|\leq C\left\{\begin{array}{l} 1,\;\;\;x\in (0,1), \\[0.2cm] x^{-\alpha-\lambda -1},\quad x\in (1,\infty) \end{array}\right.\leq\frac{C}{(1+x^2)^{\lambda +1}},\quad x\in (0,\infty),
$$
when $\alpha\geq \lambda +1$.

Since $\frac{d}{dx}(x^{-\nu}J_\nu(x))=x^{-\nu}J_{\nu +1}(x)$, $x\in (0,\infty)$ and $\nu >-1$ (\cite[(5.3.5)]{Leb}) we obtain
$$
(\phi^{\lambda,\alpha})'(x)=2^\alpha\Gamma(\alpha +1)x^{-\alpha -\lambda -1/2}J_{\alpha+\lambda+3/2}(x),\quad x\in (0,\infty).
$$
Thus, as above we obtain
$$
|(\phi^{\lambda,\alpha})'(x)|\leq C\frac{x}{(1+x^2)^{\lambda +2}},\quad x\in (0,\infty),
$$
when $\alpha\geq\lambda +2.$ In a similar way we get that
$$
|(\phi^{\lambda,\alpha})''(x)|\leq\frac{C}{(1+x^2)^{\lambda +2}},\quad x\in (0,\infty),
$$
when $\alpha\geq \lambda +3$.

We conclude that $\phi\in Z^\lambda$ provided that $\alpha\geq \lambda +3$.

It is clear that
\begin{align*} 
 & \phi_*^{\lambda.\alpha}(f)=\sup_{t>0}|B_{\lambda,t}^\alpha(f)|, \\
 &  \phi_{*,b}^{\lambda,\alpha,m}(f)(x)=\sup_{t>0}|B_{\lambda,t}^\alpha((b(\cdot)-b(x))^m f)(x)|,\quad x\in (0,\infty), \\
 & V_\rho^{\lambda,\alpha} (\{\phi_t^{\lambda,\alpha}\}_{t>0})(f)=V_\rho^\lambda(\{B_{\lambda,t}^\alpha\}_{t>0})(f),
\end{align*}
and
\begin{align*} 
 &  V_{\rho,b}^{\lambda,\alpha,m} (\{\phi_t^{\lambda,\alpha}\}_{t>0})(f)(x)=V_\rho^\lambda(\{B_{\lambda,t}^\alpha((b(\cdot)-b(x))^m f)(x)\}_{t>0}),\quad x\in (0,\infty).
\end{align*}

On the other hand, after a change of variable we get
$$
g_\psi^{\lambda,\alpha}(f)(x)=\left(\int_0^\infty|t\partial_tB_{\lambda,t}^\alpha(f)(x)|^2\frac{dt}{t}\right)^{1/2},\quad x\in (0,\infty),
$$
and
$$
g_{\psi,b}^{\lambda,\alpha,m}(f)(x)=\left(\int_0^\infty|t\partial_tB_{\lambda,t}^\alpha((b(\cdot)-b(x))^m f)(x)|^2\frac{dt}{t}\right)^{1/2},\quad x\in (0,\infty),
$$
where $\psi(x)=-(2\lambda +1)\phi^{\lambda,\alpha}(x)-x(\phi^{\lambda,\alpha})'(x)$, $x\in (0,\infty)$. Note that $\psi\in Z^\lambda$ provided that $\alpha\geq \lambda +4$.

The square function $g_\psi^\lambda$ is a Bessel version of the Stein square function introduced in \cite{Stein1}. Quantitative weighted $L^p$-inequalities for the classical Stein square function were obtained in \cite{CD}. In \cite{CDY} Chen, Duong and Yan established $L^p$-boundedness properties for Stein square function associated with operators in spaces of homogeneous type. The results in Theorem \ref{Th1.1} suppose an extension of \cite[Theorem 1.1]{CDY}.

In order to prove our results we see that the operators can be dominated by finite sums of sparse operators. This technique of controlling operators by sparse ones has been very useful in the last decade to get weighted norm inequalities. The sparse domination procedure was used at the first time by Lerner (\cite{Le1}) where an alternative proof of the $A_2$-theorem was shown. The $A_2$-theorem had been established by Hyt\"onen in \cite{Hy}. Sparse arguments have bee used to study quantitative $L^p$-weighted inequalities for several class of operators (see \cite{CPPV}, \cite{BD}, \cite{DGKLWY}, \cite{GRS}, \cite{LMR}, \cite{Le}, \cite{LeNa}, \cite{LO}, \cite{LOR}, \cite{Li}, \cite{LPRR}, \cite{Lori},  \cite{Pe} and the reference therein).

We recall the definition of systems of dyadic cubes. Let $\delta\in (0,1)$ and $0<a\leq A<\infty$. A system of dyadic cubes with parameters $(\delta,a,A)$ is a family $\displaystyle\mathcal D=\bigcup_{k\in\mathbb Z}\mathcal D_k$ where $\mathcal D_k=\{Q_\alpha^k\subset(0,\infty)\;\mbox{measurable},\;\alpha\in\mathcal A_k\}$ being $\mathcal A_k$ countable, for every $k\in\mathbb Z$, and the following properties are satisfied

\begin{enumerate}
\item[(i)] $Q_\alpha^k\cap Q_\beta^k=\emptyset$, $\alpha\neq\beta$, $\alpha,\beta\in\mathcal A_k$, $k\in\mathbb Z$ and $\displaystyle(0,\infty)=\bigcup_{\alpha\in\mathcal A_k}Q_\alpha^k$, $k\in\mathbb Z$.
\item[(ii)] If $\ell\geq k$, $\alpha\in\mathcal A_\ell$, $\beta\in\mathcal A_k$, then either $Q_\alpha^\ell\subseteq Q_\beta^k$ or $Q_\alpha^\ell\cap Q_\beta^k=\emptyset$.\newline
\item[(iii)] There exists $M\in\mathbb N$ such that  for every $k\in\mathbb Z$ and $\alpha\in\mathcal A_k$
$$
\mbox{card}\{\beta\in\mathcal A_{k+1}:\;Q_\beta^{k+1}\subseteq Q_\alpha^k\}\leq M,
$$
and $\displaystyle Q_\alpha^k=\bigcup_{Q\in\mathcal D_{k+1}, Q\subseteq Q_\alpha^k}Q$.
\item[(iv)] For every $k\in\mathbb Z$ and $\alpha\in\mathcal A_k$ there exists $x_\alpha^k\in (0,\infty)$ such that
$$
B(x_\alpha^k,a\delta^k)\subseteq Q_\alpha^k\subseteq B(x_\alpha^k,A\delta^k).
$$
Furthermore, if $\ell\geq k$ and $Q_\beta^\ell\subseteq Q_\alpha^k$, then $B(x_\beta^\ell,A\delta^\ell)\subseteq B(x_\alpha^k,A\delta^k)$.
\end{enumerate}
The set $Q_\alpha^k$ is usually named a dyadic cube of generation $k$ with center at $x_\alpha^k\in Q_\alpha^k$ and sidelength $\delta^k$.

Main properties of dyadic cubes in quasimetric spaces with the geometric doubling property can be found in \cite{HK} (see also \cite{Chr}).

Let $\mathcal D$ be a system of dyadic cubes and $0<\eta<1$. As in \cite{DGKLWY} we say that a collection $S\subset\mathcal D$ is $\eta$-sparse when  there exists $c\geq 1$ such that for every $Q\in S$ there exists a measurable set $E_Q\subset Q$ such that $m_\lambda(E_Q)\geq\eta m_\lambda(Q)$ and 
\begin{align}\label{1.1}
 \sum_{Q\in S}\mathcal X_{E_Q}(x)\leq c,\;\;\;\; x\in (0,\infty).
 \end{align}
 
 If $S\subset\mathcal D$ is $\eta$-sparse, the sparse operator $\mathfrak A_S$ is defined by
  $$
  \mathfrak A_S f(x)= \sum_{Q\in S} f_Q\mathcal X_{Q}(x),\quad x\in (0,\infty).
  $$
 Here, as above, $f_Q=\frac{1}{m_\lambda(Q)}\int_Q f(x)x^{2\lambda}dx$, $Q\in S$.
 
 Quantitative $L^p$-weighted estimates for sparse operators were established in \cite[Proposition 4.1]{Lori}.
 
 In order to use the sparse domination strategy we take into account that the operators that we study can be seen as Banach valued linear operators. Indeed, let $(E,\|\cdot\|_E)$ be a Banach space of measurable functions defined in $(0,\infty)$. We define
  $$
  \mathcal P_{\phi;E}^\lambda(f)(x)=\|f\#_\lambda\phi_{\centerdot}(x)\|_E,\quad x\in (0,\infty),
  $$
 and
  $$
  \mathcal C_{\phi,b;E}^{\lambda,m}(f)(x)=\|[(b(\cdot)-b(x))^m f]\#_\lambda\phi_{\centerdot}(x)\|_E,\quad x\in (0,\infty).
  $$
 we have that
  \begin{enumerate}
     \item[(i)] $\phi_*^\lambda=\mathcal P_{\phi;L^\infty(0,\infty)}^\lambda$ and  $\phi_{*,b}^{\lambda,m}=\mathcal C_{\phi,b;L^\infty(0,\infty)}^{\lambda,m}$.\newline
     \item[(ii)] $g_\phi^\lambda=\mathcal P_{\phi;L^2((0,\infty),\frac{dt}{t})}^\lambda$ and  $g_{\phi,b}^{\lambda,m}=\mathcal C_{\phi,b;L^2((0,\infty),\frac{dt}{t})}^{\lambda,m}$.\newline
     \item[(iii)] We define $v_\rho$ as the space of measurable functions $g:(0,\infty)\rightarrow\mathbb C$ such that
     $$
     \|g\|_{v_\rho}=\sup_{\substack{0<t_n<t_{n-1}<...<t_1 \\ n\in\mathbb N }}\left(\sum_{k=1}^{n-1}|g(t_k)-g(t_{k+1})|^\rho\right)^{1/\rho}<\infty.
     $$
     By identifying the constant functions the space $(v_\rho,\|\cdot\|_{v_\rho})$ is Banach. Furthermore,
     $$
     V_\rho^\lambda(\{\phi_t\}_{t>0})=\mathcal P_{\phi;v_\rho}^\lambda\;\;\mbox{and}\;\;V_{\rho,b}^{\lambda,m}(\{\phi_t\}_{t>0})=\mathcal C_{\phi,b;v_\rho}^{\lambda,m}.
     $$
  \end{enumerate}
 In the proof of Theorems \ref{Th1.1} and \ref{Th1.2}, that we write in the next sections, the results in \cite{DGKLWY} and \cite{Lori} concerning to sparse domination on spaces of homogeneous type in a vector valued setting play an important role.
 
 Throughout this paper $C$ will always represents a positive constant that can change in each occurrence.

\section{Proof of Theorem \ref{Th1.1}}\label{S2}

In \cite{Le} a sparse domination principle was given for an operator $T$ by using the grand maximal truncated operator $\mathcal M_T$ of $T$ defined by
$$
\mathcal M_T(f)(x)=\sup_{Q\ni x}\|T(f\mathcal X_{\mathbb R^n\setminus(3Q)})\|_{L^\infty(Q)},\quad X\in\mathbb R^n,
$$
that can be established as follows.

\begin{thm}\label{Th2.1}(\cite[Theorem 4.2]{Le}).
Let $1\leq q\leq r<\infty$. Assume that $T$ is a sublinear operator that is bounded from $L^q(\mathbb R^n)$ into $L^{q,\infty}(\mathbb R^n)$ and such that $\mathcal M_T$ is  bounded from $L^r(\mathbb R^n)$ into $L^{r,\infty}(\mathbb R^n)$. Then, for every compactly supported $f\in L^r(\mathbb R^n)$, there exists a sparse family $S$ such that
$$
|Tf(x)|\leq C\sum_{Q\in S}\langle f\rangle_{r,Q}\mathcal X_Q(x),\quad \mbox{a.e. }x\in\mathbb R^n,
$$
where $C=C_{n,q,r}\left(\|T\|_{L^q(\mathbb R^n)\rightarrow L^{q,\infty}(\mathbb R^n)}+\|\mathcal M_T\|_{L^r(\mathbb R^n)\rightarrow L^{r,\infty}(\mathbb R^n)}\right)$. 

Here, $\langle f\rangle_{r,Q}=\left(\frac{1}{|Q|}\int_Q |f(x)|^r dx\right)^{1/r}$, $Q\in S$.
\end{thm}
In \cite{LO} Theorem \ref{Th2.1} was improved by weakening some of the hypotheses. The operator $\mathcal M_T$ is replaced by the one $\mathcal M_{T,\alpha}^\#$ defined for every $\alpha >0$ by
$$
\mathcal M_{T,\alpha}^\#(f)(x)=\sup_{Q\ni x}\;\esssup_{x',x''\in Q}|T(f\mathcal X_{\mathbb  R^n\setminus(\alpha Q)})(x')-T(f\mathcal X_{\mathbb R^n\setminus(\alpha Q)})(x'')|,\quad x\in\mathbb R^n,
$$
(see \cite[Theorem 1.1]{LO}). 

Note that the operators $\mathcal M_T$ and $\mathcal M_{T,3}^\#$ are connected as follows. Let $Q$ be a cube in $\mathbb R^n$ and $x\in Q$. We have that
\begin{align*}
|T(f\mathcal X_{\mathbb  R^n\setminus(3 Q)})(y)| & \leq |T(f\mathcal X_{\mathbb  R^n\setminus(3 Q)})(y)-T(f\mathcal X_{\mathbb R^n\setminus(3 Q)})(z)| \\
& + |Tf(z)|+|T(f\mathcal X_{3Q})(z)|, \;\;\;\;y,z\in Q.
\end{align*}
Then,
\begin{align*}
|T(f\mathcal X_{\mathbb  R^n\setminus (3 Q)})(y)| & \leq \esssup_{y,z\in Q}|T(f\mathcal X_{\mathbb  R^n\setminus(3 Q)})(y)-T(f\mathcal X_{\mathbb R^n\setminus(3 Q)})(z)| \\
& + \inf_{z\in Q}\left(|Tf(z)|+|T(f\mathcal X_{3Q})(z)|\right) \\
& \leq\esssup_{y,z\in Q}|T(f\mathcal X_{\mathbb  R^n\setminus(3 Q)})(y)-T(f\mathcal X_{\mathbb R^n\setminus(3 Q)})(z)| \\
& + C\left[\left(\frac{1}{|Q|}\int_Q |T(f\mathcal X_{3Q})(z)|^{1/2}dz\right)^2+\left(\frac{1}{|Q|}\int_Q |Tf(z)|^{1/2}dz\right)^2\right] \\
& \leq\esssup_{y,z\in Q}|T(f\mathcal X_{\mathbb  R^n\setminus(3 Q)})(y)-T(f\mathcal X_{\mathbb R^n\setminus(3 Q)})(z)| \\
& + C\left[\left(\mathcal M[(T(f\mathcal X_{3Q}))^{1/2}]\right)^2(x)+\left(\mathcal M[(Tf)]^{1/2}\right)^2(x)\right],
\end{align*}
where $\mathcal M$ denotes the Hardy-Littlewood maximal function. We get
$$
\mathcal M_T(f)(x)\leq C\left(\mathcal M_{T,3}^\#(f)(x)+\left(\mathcal M[(Tf)]^{1/2})(x)\right)^2+\left(\mathcal M[(T(f\mathcal X_{3Q}))^{1/2}](x)\right)^2\right).
$$
Suppose that $T$ is bounded from $L^1(\mathbb R^n)$ into $L^{1,\infty}(\mathbb R^n)$. By using Kolmogorov inequality (\cite[p. 91]{Gra}) we deduce that
$$
\mathcal M_T(f)(x)\leq C\left(\mathcal M_{T,3}^\#(f)(x)+\left(\mathcal M(|Tf|^{1/2})(x)\right)^2+\mathcal M(|f|)(x)\right).
$$
By using \cite[Lemma 3.2]{NTV} we conclude that $\mathcal M_T$ is bounded from $L^1(\mathbb R^n)$ into $L^{1,\infty}(\mathbb R^n)$ provided that $\mathcal M_{T,3}^\#$  is bounded from $L^1(\mathbb R^n)$ into $L^{1,\infty}(\mathbb R^n)$.

\noindent On the other hand it is clear that $\mathcal M_{T,3}^\#$ is bounded from $L^1(\mathbb R^n)$ into $L^{1,\infty}(\mathbb R^n)$ when $\mathcal M_T$ is bounded from $L^1(\mathbb R^n)$ into $L^{1,\infty}(\mathbb{R}^n)$.

In our Bessel contexts we consider the spaces of homogeneous type $((0,\infty),|\cdot|,m_\lambda)$. A version of Theorem \ref{Th2.1} for spaces of homogeneous type in vector valued settings was established in \cite[Theorem 1.1]{Lori}. We will prove that the operators under consideration in Theorem \ref{Th1.1} are bounded from $L^1((0,\infty),m_\lambda)$ into $L^{1,\infty}((0,\infty),m_\lambda)$. Then the above comments concerning to $\mathcal M_T$ and  $\mathcal M_{T,3}^\#$ also works in the $((0,\infty),|\cdot|,m_\lambda)$ context.

In the sequel for each $x,r \in (0,\infty)$, we write $B(x,r)$ to denote the interval $(x-r,x+r) \cap (0,\infty)$,  and $cB(x,r)=B(x,cr)$, $c>0$.

\subsection{Proof of Theorem \ref{Th1.1} for $\phi_*^\lambda$}\label{S2.1}

Assume That $\phi$ satisfies (i) (see the definition of $Z^\lambda$). According to \cite[(2.4)]{YY} we have that
$$
|_\lambda\tau_x(\phi_t)(y)|\leq\frac{C}{m_\lambda(B(x,t))+m_\lambda(B(x,|x-y|))}\frac{t}{t+|x-y|},\quad t,x,y\in (0,\infty).
$$
Then,
\begin{align*}
|(f\#_\lambda\phi)(x)| & \leq \sum_{k=1}^\infty\int_{B(x,2^kt)\setminus B(x,2^{k-1}t)}|_\lambda\tau_x(\phi_t)(y)||f(y)|y^{2\lambda}dy \\
& +\int_{B(x,t)}|_\lambda\tau_x(\phi_t)(y)||f(y)|y^{2\lambda}dy \\
& \leq C\sum_{k=0}^\infty \frac{1}{m_\lambda(B(x,2^kt))}\frac{1}{2^k}\int_{B(x,2^kt)}|f(y)|y^{2\lambda}dy \\
& \leq C\mathcal M_\lambda(f)(x),\;\;\;\;t,x\in (0,\infty).
\end{align*}
Here $\mathcal M_\lambda$ represents the corresponding Hardy-Littlewood maximal operator defined by
$$
\mathcal M_\lambda(f)(x)=\sup_{r>0}\frac{1}{m_\lambda(B(x,r))}\int_{B(x,r)}|f(y)|y^{2\lambda}dy,\quad f\in L^1_{loc}((0,\infty),m_\lambda).
$$
Hence,
$$
\phi_*^\lambda(f)(x)\leq C\mathcal M_\lambda(f)(x),\quad x\in (0,\infty).
$$
Since
$$
\|\mathcal M_\lambda(f)\|_{L^p((0,\infty),w,m_\lambda)}\leq C[w]_{A_p}^{\frac{1}{p-1}}\|f\|_{L^p((0,\infty),w,m_\lambda)},
$$
for every $f\in L^p((0,\infty),w,m_\lambda)$ (see, for instance, \cite[p. 890]{BD}) we conclude the proof of Theorem \ref{Th1.1} for $\phi_*^\lambda$. Note that here it is sufficient to assume that $\phi$ satisfies (i).

\subsection{Proof of Theorem \ref{Th1.1} for $g_\phi ^\lambda$}\label{S2.2}
We firstly establish $L^p$-boundedness properties for the sublinear operator $g^\lambda_\phi$.
\begin{prop} \label{propo 2.1}
Let $\lambda >0$. Suppose that $\phi$ is a differentiable function on $(0,\infty)$ satisfying the following three conditions:
\begin{enumerate}
\item[(a)] $|\phi(x)| \leq C(1+x^2)^{-\lambda -1}$, $x \in (0,\infty)$;
\item[(b)] $\displaystyle \int_0^\infty \left|u^{2\lambda +2}\phi'(u)\right|^2 \frac{du}{u} <\infty$;
\item[(c)] $\displaystyle \int_0^\infty \left|h_\lambda (\phi)(u)\right|^2\frac{du}{u} < \infty$.
\end{enumerate}
Then, the operator $g_\phi^\lambda$ is bounded from $L^p((0,\infty), m_\lambda)$ into itself, for every $1<p<\infty$, and from $L^1((0,\infty),m_\lambda)$ into $L^{1,\infty}((0,\infty),m_\lambda)$.
\end{prop}
\begin{proof}
We have that
$$
g^\lambda_\phi(f)(x) =\|f \#_\lambda \phi_\centerdot(x)\|_{L^2((0,\infty),\frac{dt}{t})}, \quad x \in (0,\infty).
$$
In order to see the $L^p$-boundedness properties for $g^\lambda_\phi$ we use vector valued singular integral theory (\cite{RRT}).

We define
$$
\mathfrak{G}^\lambda_{\phi,t}(f)(x)= \int_0^\infty \mathfrak{G}^\lambda_{\phi,t}(x,y) f(y)y^{2\lambda}dy,\quad  x,t \in (0,\infty),
$$
where $\mathfrak{G}^\lambda_{\phi,t}(x,y)=\,_\lambda\tau_x (\phi_t)(y)$, $x,y,t \in (0,\infty)$.

By using Minkowski inequality we get, for every $x,y\in (0,\infty)$
$$
\left\|\mathfrak{G}^\lambda_{\phi,\centerdot}(x,y)\right\|_{L^2\left((0,\infty),\frac{dt}{t}\right)} \leq C \int_0^\pi \sin^{2\lambda-1}\theta\left(\int_0^\infty \left|\frac{1}{t^{2\lambda+1}}\phi\left(\frac{\sqrt{x^2+y^2-2xy\cos\theta}}{t}\right)\right|^2 \frac{dt}{t}\right)^{\frac{1}{2}} d\theta . 
$$
We have that
\begin{eqnarray*}
\int_0^\infty \left|\frac{1}{t^{2\lambda+1}}\phi\left(\frac{\sqrt{x^2+y^2-2xy\cos\theta}}{t}\right)\right|^2\frac{dt}{t}&=& \frac{1}{(x^2+y^2-2xy\cos\theta)^{2\lambda+1}}\int_0^\infty |u^{2\lambda +1}\phi(u)|^2 \frac{du}{u} \\
&\hspace{-2cm}\leq & \hspace{-1cm}\frac{C}{(x^2+y^2-2xy\cos\theta)^{2\lambda +1}}, \quad x,y \in (0,\infty),\,\theta\in (0,\pi),
\end{eqnarray*}
because $\phi$ satisfies (a).

According to \cite[Lemma 3.1]{BCN} we obtain
\begin{eqnarray}\label{2.1}
\left\| \mathfrak{G}^\lambda_{\phi,\centerdot}(x,y)\right\|_{L^2((0,\infty),\frac{dt}{t})} &\leq& C\int^\pi_0 \frac{\sin^{2\lambda -1}\theta}{(x^2+y^2-2xy\cos\theta)^{\lambda+1/2}}d\theta\nonumber\\
& \leq &\frac{C}{m_\lambda(B(x,|x-y|))}, \quad x,y \in(0,\infty),\;x \not=y.
\end{eqnarray}

We can write for each $x,y,t \in (0,\infty)$,
$$
\partial_x \mathfrak{G}^\lambda_{\phi,t}(x,y)=\frac{\Gamma(\lambda + \frac{1}{2})}{\sqrt\pi \Gamma(\lambda)t^{2\lambda+1}}\int_0^\pi \phi'\left( \frac{\sqrt{x^2+y^2-2xy\cos \theta}}{t}\right)\frac{(x-y\cos\theta)\sin^{2\lambda-1}\theta}{t\sqrt{x^2+y^2-2xy\cos\theta}}d\theta.
$$
Since $|a+bs| \leq (a^2+b^2+2abs)^{1/2}$, $a,b \in (0,\infty)$ and $s \in (-1,1)$, it follows that
$$
\left|\partial_x\mathfrak{G}^\lambda_{\phi,t}(x,y)\right|\leq \frac{C}{t^{2\lambda +2}}\int_0^\pi \left|\phi'\left( \frac{\sqrt{x^2+y^2-2xy\cos \theta}}{t}\right)\right|\sin^{2\lambda -1}\theta d\theta,\quad x,y,t\in (0,\infty).
$$
By proceeding as above by using (b), we get
\begin{align*}
\left\| \partial_x \mathfrak{G}_{\phi,\centerdot}(x,y)\right\|_{L^2((0,\infty),\frac{dt}{t})} &\\
&\hspace{-1cm}\leq  C\int_0^\pi\sin^{2\lambda -1}\theta\left( \int_0^\infty \left|\frac{1}{t^{2\lambda+2}}\phi'\left(\frac{\sqrt{x^2+y^2-2xy\cos \theta}}{t}\right)\right|^2\frac{dt}{t}\right)^{\frac{1}{2}}d\theta\\
&\hspace{-1cm}\leq  C\int_0^\pi \frac{\sin^{2\lambda-1}\theta}{(x^2+y^2-2xy\cos \theta)^{\lambda+1}}\left(\int_0^\infty |u^{2\lambda+2}\phi'(u) |^2\frac{du}{u}\right)^{\frac{1}{2}}\\
&\hspace{-1cm}\leq  C\int_0^\pi \frac{\sin^{2\lambda-1}\theta}{(x^2+y^2-2xy\cos \theta)^{\lambda+1}}d\theta,\quad x,y \in (0,\infty).
\end{align*}
According to \cite[Lemma 3.1]{BCN} we obtain
\begin{equation}\label{2.2}
\left\|\partial_x \mathfrak{G}^\lambda_{\phi,\centerdot}(x,y)\right\|_{L^2((0,\infty), \frac{dt}{t})} \leq  \frac{C}{|x-y|m_\lambda(B(x,|x-y|))},\quad x,y\in (0,\infty),\,x\not=y.
\end{equation}

By symmetry we also get
\begin{equation}\label{2.3}
\left\|\partial_y \mathfrak{G}^\lambda_{\phi,\centerdot}(x,y)\right\|_{L^2((0,\infty), \frac{dt}{t})} \leq  \frac{C}{|x-y|m_\lambda(B(x,|x-y|))},\quad x,y\in (0,\infty),\,x\not=y.
\end{equation}
Let $N\in \mathbb{N}$. We consider for every smooth function $f$ with compact support in $(0,\infty )$ 
\begin{equation}\label{2.4}
\mathcal{G}^\lambda_{\phi,N}(f)(x) = \int_0^\infty \mathfrak{G}^\lambda_{\phi,\centerdot}(x,y) f(y) y^{2\lambda} dy,\quad  x\in (0,\infty)\setminus {\rm supp } f,
\end{equation}
where the integral is understood in the $L^2((\frac{1}{N},N),\frac{dt}{t})$-Bochner sense.  Note that, given a smooth function $f$ with compact support in $(0,\infty)$ and $ x\in (0,\infty)\setminus {\rm supp } f$, the function
$$
F_x(y)=\mathfrak{G}^\lambda_{\phi,\centerdot}(x,y) f(y) y^{2\lambda},\quad y\in (0,\infty),
$$
is continuous from $(0,\infty)$ into $L^2((\frac{1}{N},N),\frac{dt}{t})$. Then, since $L^2((\frac{1}{N},N),\frac{dt}{t})$ is a separable space we obtain that $F_x$, $x\not\in {\rm \supp }f$, is a strongly measurable function. Furthermore,  according to (\ref{2.1}) the integral in (\ref{2.4}) is $L^2((\frac{1}{N},N),\frac{dt}{t})$-norm convergent provided that $x\notin\supp f$.

Let $f$ be a smooth function with compact support in $(0,\infty)$.  Then, $\mathcal{G}^\lambda_{\phi,N}(f)(x) \in L^2((\frac{1}{N},N),\frac{dt}{t})$, for every $x\in (0,\infty)\setminus \supp f$.

Consider $h \in L^2((\frac{1}{N},N),\frac{dt}{t})$. We have that
$$
\int_{1/N}^N h(t) \left[\mathcal{G}_{\phi,N}^\lambda(f)(x)\right](t)\frac{dt}{t} = \int_0^\infty\int_{1/N}^N h(t) \mathfrak{G}^\lambda_{\phi,t}(x,y) \frac{dt}{t}f(y) y^{2\lambda}dy,\quad x\in (0,\infty)\setminus\supp f.
$$

By using again (\ref{2.1}) we get
\begin{eqnarray*}
\int_0^\infty\int_{1/N}^N|h(t)||\mathfrak{G}^\lambda_{\phi,t}(x,y)|\frac{dt}{t}|f(y)|y^{2\lambda} dy && \\
&&\hspace{-5cm}\leq \|h\|_{L^2((\frac{1}{N},N),\frac{dt}{t})}\int_0^\infty \|\mathfrak{G}^\lambda_{\phi,\centerdot}(x,y)\|_{L^2((0,\infty),\frac{dt}{t})}|f(y)|y^{2\lambda}dy \\
&&\hspace{-5cm}\leq \|h\|_{L^2((\frac{1}{N},N),\frac{dt}{t})}\int_0^\infty \frac{|f(y)|y^{2\lambda}}{m_\lambda(B(x,|x-y|))} dy < \infty,\quad x \in (0,\infty) \setminus \supp f.
\end{eqnarray*}
It follows that
$$
\int_{1/N}^N h(t)[\mathcal{G}^\lambda_{\phi,N}(f)(x)](t)\frac{dt}{t} = \int_{1/N}^N h(t)\mathfrak{G}^\lambda_{\phi,t}(f)(x)\frac{dt}{t},\quad  x\in (0,\infty)\setminus\supp f.
$$
Hence, for every $x\in (0,\infty)\setminus \supp f$,
$$
[\mathcal{G}^\lambda_{\phi,N}(f)(x)](t) = \mathfrak{G}^\lambda_{\phi,t}(f)(x),\quad \mbox{a.e. }t\in \big(\frac{1}{N},N\big).
$$

We now observe that the function
$$
G(x,y,t)=\mathfrak{G}^\lambda_{\phi,t}(x,y) f(y) y^{2\lambda},\quad x,y,t\in (0,\infty),
$$
is continuous. Then, if $x_0\in (0,\infty )$, we have that
$$
\lim_{x\rightarrow x_0}\Big\| \mathfrak{G}^\lambda_{\phi,\centerdot}(f)(x)-\mathfrak{G}^\lambda_{\phi,\centerdot}(f)(x_0)\Big\|_{L^2((\frac{1}{N},N),\frac{dt}{t})}=0.
$$
So, as above we deduce that $\mathfrak{G}^\lambda_{\phi,\centerdot}(f): (0,\infty)\rightarrow L^2((\frac{1}{N},N),\frac{dt}{t})$ is a strongly measurable function.

By using Plancherel formula for Hankel transformation we deduce that
\begin{align*}
\left\|\mathfrak{G}^\lambda_{\phi,\centerdot}(f)\right\|^2_{L^2_{L^2((\frac{1}{N},N)\frac{dt}{t})}((0,\infty),m_\lambda)}&=\left\|\|\mathfrak{G}^\lambda_{\phi,\centerdot}(f)\|_{L^2((\frac{1}{N},N)\frac{dt}{t})}\right\|^2_{L^2((0,\infty),m_\lambda)}\\
&\leq  \int_0^\infty\int_0^\infty\left|\mathfrak{G}^\lambda_{\phi,t}(f)(x)\right|^2\frac{dt}{t}x^{2\lambda}dx\\
&= \int_0^\infty\int_0^\infty\left|(f\#_\lambda{\phi}_t)\right|^2 x^{2\lambda}dx\frac{dt}{t}\\
&= \int_0^\infty \int_0^\infty\left|h_\lambda(f)(x)\right|^2\left|h_\lambda(\phi_t)(x)\right|^2 x^{2\lambda}dx\frac{dt}{t}\\
&= \int_0^\infty \left|h_\lambda(f)(x)\right|^2\int_0^\infty \left|h_\lambda(\phi)(xt)\right|^2\frac{dt}{t}x^{2\lambda}dx\\
&= \int_0^\infty \left|h_\lambda(f)(x)\right|^2\int_0^\infty \left|h_\lambda(\phi)(u)\right|^2\frac{du}{u}x^{2\lambda}dx\\
&=\|f\|^2_{L^2((0,\infty),m_\lambda)}\int_0^\infty\left|h_\lambda(\phi)(u)\right|^2\frac{du}{u}.
\end{align*}
By property (c) we conclude that the operator
$$
\begin{array}{ccc}
\mathfrak{G}^\lambda_\phi:L^2((0,\infty),m_\lambda) & \longrightarrow & L^2_{L^2((\frac{1}{N},N),\frac{dt}{t})}((0,\infty),m_\lambda)\\
f& \longrightarrow & \mathfrak{G}^\lambda_{\phi}(f) 
\end{array}
$$
is bounded, where 
$$
\begin{array}{rccrccc}
\mathfrak{G}^\lambda_{\phi}(f):&(0,\infty) & \longrightarrow & L^2((\frac{1}{N},N),\frac{dt}{t}) &&&\\
&x   &  \longrightarrow &\mathfrak{G}^\lambda_{\phi}(f)(x):&(\frac{1}{N},N) &\longrightarrow &\mathbb R\\
&   &         & &t &\longrightarrow &\mathfrak{G}^\lambda_{\phi,t}(f)(x).
\end{array}
$$

By using vector valued Calder\'on-Zygmund theory (\cite{RRT}) we deduce that the operator $g^\lambda_{\phi,N}$ given by
$$
g^\lambda_{\phi,N}(f)(x)=\|f\#_\lambda \phi _{\centerdot}(x)\|_{L^2((\frac{1}{N},N),\frac{dt}{t})},\quad x\in (0,\infty ),
$$
can be extended from $L^2((0,\infty),m_\lambda) \cap L^p((0,\infty),m_\lambda)$ to $L^p((0,\infty),m_\lambda)$ as a bounded operator from $L^p((0,\infty),m_\lambda)$ to
\begin{enumerate}
\item $L^p((0,\infty),m_\lambda)$, for every $1<p<\infty$;
\item $L^{1,\infty}((0,\infty),m_\lambda)$, when $p=1$,
\end{enumerate}
with constants independent of $N$. This allows us to use monotone convergence theorem to get the same $L^p$-boundedness properties for $g_\phi^\lambda$. 
\end{proof}

We consider the operator
$$
\mathcal M^{\lambda,\#}_\phi(f)(x)= \sup_{\substack{Q\ni x\\ Q\,\text{interval}}} \esssup_{y_1,y_2 \in Q}\left|g^\lambda_\phi(f\chi_{(0,\infty)\setminus(3Q)})(y_1) - g^\lambda_\phi(f\chi_{(0,\infty)\setminus(3Q)})(y_2)\right|,\;x\in (0,\infty)
$$
\begin{prop}\label{propo 2.2}
Let $\lambda >0$. Suppose that $\phi$ is a differentiable function on $(0,\infty)$ such that $\int_0^\infty |u^{2\lambda +2}\phi '(u)|^2\frac{du}{u}<\infty$. Then, the operator $\mathcal M_\phi^{\lambda,\#}$ is bounded from $L^p((0,\infty),m_\lambda)$ into itself, for every $1<p<\infty$, and from $L^1((0,\infty),m_\lambda)$ into $L^{1,\infty}((0,\infty),m_\lambda)$.
\end{prop}
\begin{proof} Let $x\in (0,\infty)$. We choose an interval $Q=B(x_Q,r_Q) \subset (0,\infty)$, with $x_Q,r_Q\in (0,\infty)$, and  such that $x\in Q$. By using Minkowski inequality and (\ref{2.2}) we obtain
\begin{eqnarray*}
\left|g^\lambda_\phi(f\chi_{(0,\infty)\setminus(3Q)})(y_1) - g^\lambda_\phi(f\chi_{(0,\infty)\setminus(3Q)})(y_2)\right| & &\\
& &\hspace{-5cm} \leq \left( \textcolor{black}{\int_0^\infty \left|\mathfrak G^\lambda_{\phi,t}\left( f\chi_{(0,\infty)\setminus(3Q)}\right)(y_1) - \mathfrak G^\lambda_{\phi,t}\left( f\chi_{(0,\infty)\setminus(3Q)}\right)(y_2)\right|^2\frac{dt}{t}}\right)^{1/2}\\
& &\hspace{-5cm} \leq \int_{(0,\infty)\setminus(3Q)} |f(y)|y^{2\lambda}\left(\int_0^\infty \left|\,_\lambda \tau_{y_1}(\phi_t)(y)-\,_\lambda\tau_{y_2}(\phi_t)(y)\right|^2\frac{dt}{t}\right)^{\frac{1}{2}}dy\\
& &\hspace{-5cm} = \int_{(0,\infty)\setminus(3Q)} |f(y)|y^{2\lambda}\left(\int_0^\infty\left|\int_{y_1}^{y_2} \partial_z\left(\,_\lambda \tau_z(\phi_t)(y)\right)dz\right|^2\frac{dt}{t}\right)^{\frac{1}{2}}dy\\
& &\hspace{-5cm} \leq \int_{(0,\infty)\setminus(3Q)} |f(y)|y^{2\lambda}\left|\int_{y_1}^{y_2}\left(\int_0^\infty\left| \partial_z\left(\,_\lambda\tau_z(\phi_t)(y)\right)\right|^2\frac{dt}{t}\right)^{\frac{1}{2}}dz\right|dy\\
& &\hspace{-5cm} \leq C\int_{(0,\infty)\setminus(3Q)} |f(y)|y^{2\lambda}\left|\int_{y_1}^{y_2}\frac{dz}{|z-y|m_\lambda(B(z,|z-y|))}\right|dy\\
& &\hspace{-5cm} \leq C|y_1-y_2|\int_{(0,\infty)\setminus(3Q)} \frac{|f(y)|y^{2\lambda}}{|y_1-y|m_\lambda(B(y_1,|y-y_1|))}dy,\;\; y_1,y_2\in Q.
\end{eqnarray*}

Let $y_1\in Q$. We define $\widetilde Q = B(y_1,r_Q)$. Then, by taking into account that $m_\lambda$ is doubling, we obtain
\begin{eqnarray*}
\left|g_\phi^\lambda\left(f\chi_{(0,\infty)\setminus(3Q)}\right)(y_1) - g_\phi^\lambda\left(f\chi_{(0,\infty)\setminus(3Q)}\right)(y_2)\right| &&\\
& &\hspace{-5cm}\leq Cr_Q\sum^\infty_{k=1}\int_{2^{k+1}\widetilde Q \setminus 2^k\widetilde Q}\frac{|f(y)|}{|y_1-y|m_\lambda(B(y_1,|y-y_1|))}y^{2\lambda}dy\\
& &\hspace{-5cm}\leq Cr_Q\sum^\infty_{k=1}\frac{1}{2^kr_Q} \frac{1}{m_\lambda(B(y_1,2^{k+1}r_Q))}\int_{2^{k+1}\widetilde Q }|f(y)|y^{2\lambda}dy \\
& &\hspace{-5cm}\leq C\mathcal{M}_\lambda(f)(x),\quad  y_2 \in Q.
\end{eqnarray*}
We get
$$
\mathcal M_\phi^{\lambda,\#}(f)(x) \leq C\mathcal M_\lambda (f)(x).
$$
The $L^p$-boundedness properties of $\mathcal M^{\lambda,\#}_\phi$ can be deduced now from the corresponding ones of $\mathcal M_\lambda$.
\end{proof}

By using \cite[Theorem 1.1]{Lori}, the results in Proposition \ref{propo 2.1} and \ref{propo 2.2} lead to that, for a certain $\eta \in (0,1)$ the following property holds: for every $f \in L^1((0,\infty), m_\lambda)$ with compact support there exists an $\eta$-sparse collection of cubes $S$ in $(0,\infty)$ such that $c=1$ in (\ref{1.1}) and
$$
g_\phi^\lambda(f)(x) \leq CC_{\phi,\lambda} \sum_{Q \in S}|f|_Q\chi_Q (x),\;\;x\in (0,\infty),
$$
where
$$
C_{\phi,\lambda}=\|g_\phi^\lambda\|_{L^1((0,\infty),m_\lambda) \rightarrow L^{1,\infty}((0,\infty),m_\lambda)}+\left\|\mathcal M_\phi^{\lambda,\#}\right\|_{L^1((0,\infty),m_\lambda) \rightarrow L^{1,\infty}((0,\infty),m_\lambda)},
$$
and $C$ depends on $S$. Here $|f|_Q =\frac{1}{m_\lambda(Q)}\int_Q |f(y)|y^{2\lambda}dy$, $Q \in S$.
The proof can be finished by using \cite[Proposition 4.1]{Lori}.

\subsection{Proof of Theorem \ref{Th1.1} for $V_\rho ^\lambda (\{\phi _t\}_{t>0})$}\label{S2.3}

In the next proposition we establish the $L^p$-boundedness properties of the $\rho$-variation operator $V_\rho^\lambda(\{\phi_t\}_{t>0})$. In order to prove this proposition we can not apply the vector-valued Calder\'on-Zygmund theory because we do not now if $V_\rho^\lambda (\{\phi_t\}_{t>0})$ is bounded from $L^q((0,\infty),m_\lambda)$ into itself for some $1<q<\infty$.
\begin{prop}\label{propo 2.3}
Let $\lambda >0$ and $\rho >2$. Assume that $\phi \in C^2(0,\infty)$ and let 
$$
\Phi(z)=\int_0^\infty u^{\lambda -1}\phi (\sqrt{z^2+u})du,\quad z\in \mathbb{R}.
$$
Suppose that the following properties hold:
\begin{enumerate}
\item [(a)] $\displaystyle\int_0^\infty |\phi^{(k)}(x)| x^{2\lambda +k}dx < \infty$, $k=0,1,2$;
\item [(b)] $\displaystyle\int_{\frac{1}{2}}^2 \frac{v^\lambda}{|1-v|}\int_0^\infty s^{k-1/2}\int_{\frac{\pi^2vs}{4(1-v)^2}} u^{\lambda-1}\frac{|\phi^{(k)}(\sqrt{s+u})|}{(s+u)^{k/2}}dudsdv < \infty$, $k=0,1$;
\item [(c)] $\displaystyle\int_0^\infty z^2\int_0^\infty u^{\lambda-1}|\phi'(\sqrt{z^2+u})|\frac{dudz}{\sqrt{z^2+u}} < \infty$;
\item [(d)] $\displaystyle \lim_{|z|\rightarrow \infty}\Phi(z)=0$;
\item[(e)] $\displaystyle \int_0^\infty (|\Phi(u)|+u|\Phi'(u)|)du<\infty$.
\end{enumerate}
Then, the operator $V_\rho^\lambda(\{\phi_t\}_{t>0})$ is bounded from $L^p((0,\infty),m_\lambda)$ into itself, for every $1<p<\infty$, and from
$L^1((0,\infty),m_\lambda)$ into $L^{1,\infty}((0,\infty),m_\lambda)$
\end{prop}
\begin{proof}
We define $K_t(x,y) =\,_\lambda \tau_x(\phi_t)(y)$, $t,x,y\in(0,\infty)$. We decompose $K_t(x,y)$ as follows
$$
 K_t(x,y)=K_{t,1}(x,y)+K_{t,2}(x,y),\quad t,x,y \in (0,\infty),
$$
being
$$
K_{t,1}(x,y)=\frac{\Gamma(\lambda + \frac{1}{2})}{\sqrt \pi \Gamma(\lambda)t^{2\lambda+1}}\int _0^{\frac{\pi}{2}} \phi\left(\frac{\sqrt{x^2+y^2-2xy\cos\theta}}{t}\right) \sin^{2\lambda -1}\theta d\theta,\quad t,x,y \in (0,\infty).
$$
We consider, for $j=1,2$, 
$$
T_{t,j}(f)(x) =\int_0^\infty K_{t,j}(x,y) f(y)y^{2\lambda}dy,\;\; t,x \in (0,\infty).
$$
It follows that
$$
V^\lambda_\rho(\{\phi_t\}_{t>0})(f)\leq V_{\rho}(\{T_{t,1}\}_{t>0})(f) + V_\rho(\{T_{t,2}\}_{t>0})(f),
$$
where $V_\rho(\{T_{t,j}\}_{t>0})$, $j=1,2$, are defined in the obvious way. We are going to study $V_\rho(\{T_{t,j}\}_{t>0})$, $j=1,2$.

Suppose that $t_1>t_2> \ldots > t_k>0$, with $k  \in  \mathbb N$. We have that

\begin{align*}
\left( \sum_{j=1}^{k-1}\left|T_{t_j,2}(f)(x)- |T_{t_{j+1,2}}(f)(x)\right|^\rho\right)^{\frac{1}{\rho}} &= \left(\sum_{j=1}^{k-1}\left|\int_{t_{j+1}}^{t_j} \partial_tT_{t,2}(f)(x)dt\right|^{\rho}\right)^{\frac{1}{\rho}} \\
&\hspace{-2cm}\leq  C\sum_{j=1}^{k-1}\int_{t_{j+1}}^{{t_j}}\left|\partial_tT_{t,2}(f)(x)\right|dt\leq C\int_0^\infty\left|\partial_tT_{t,2}(f)(x)\right|dt \\
&\hspace{-2cm}\leq C\int_0^\infty |f(y)|y^{2\lambda}\int_0^\infty \left|\partial_tK_{t,2}(x,y)\right|dtdy, \quad x\in(0,\infty).
\end{align*}
We get
$$
V_\rho(\{T_{t,2}\}_{t>0})(f)(x) \leq C\int_0^\infty  \mathbb K_2(x,y) |f(y)|y^{2\lambda}dy,\quad x\in (0,\infty)
$$
where
$$
\mathbb K_2(x,y)= \int_0^\infty \left|\partial_tK_{t,2}(x,y)\right|dt,\quad x,y  \in (0,\infty).
$$
Since (a) for $k=0,1$ hold we obtain
\begin{eqnarray*}
|\mathbb K_2(x,y)| &\leq& C\int_0^\infty \frac{1}{t^{2\lambda +2}} \int_{\frac{\pi}{2}}^\pi(|\phi(u)|+|\phi'(u)|u)_{|{u=\frac{\sqrt{x^2+y^2-2xy\cos\theta}}{t}}}\sin^{2\lambda-1}\theta d\theta dt\\
&\leq& C\int_{\frac{\pi}{2}}^\pi \frac{\sin^{2\lambda -1}\theta}{(\sqrt{x^2+y^2-2xy\cos\theta})^{2\lambda+1}}\int_0^\infty(|\phi(u)|+|\phi'(u)|u)u^{2\lambda}du d\theta\\
&\leq & \frac{C}{(x^2+y^2)^{\lambda+\frac{1}{2}}} \leq C\left\{\begin{array}{ll}
                                                               y^{-2\lambda-1}, & 0<x\leq y < \infty,\\[0,2cm]
                                                               x^{-2\lambda-1}, & 0<y\leq x <\infty.
                                                               \end{array}\right.
\end{eqnarray*}
We consider the Hardy type operators defined by
$$
H_0^\lambda(g)(x)= \frac{1}{x^{2\lambda +1}}\int_0^x g(y) y^{2\lambda}dy,\quad x \in (0,\infty),
$$
and
$$
H_\infty(g)(x)=\int_x^\infty g(y)\frac{dy}{y}, \quad x \in (0,\infty).
$$
According to \cite[Lemmas 2.6 and 2.7]{BHNV}, $H_0^\lambda$ and $H_\infty$ are bounded from $L^p((0,\infty),m_\lambda)$ into itself for every $1<p<\infty$, and from $L^1((0,\infty),m_\lambda)$ into $L^{1,\infty}((0,\infty),m_\lambda)$.

We deduce that
$$
V_\rho(\{T_{t,2}\}_{t>0})(f)(x) \leq C(H_0^\lambda (|f|)(x)+H_\infty (|f|)(x)),\quad x\in(0,\infty ).
$$

We now study $V_\rho(\{T_{t,1}\}_{t>0})$. We write
\begin{eqnarray*}
T_{t,1}(f)(x) &=& \left(\int_0^{\frac{x}{2}} + \int_{\frac{x}{2}}^{2x}+\int_{2x}^\infty\right)K_{t,1}(x,y)f(y)y^{2\lambda}dy\\
&=& \sum_{j=1}^3 T_{t,1,j}(f)(x),\;\; t,x \in (0,\infty).
\end{eqnarray*}
By proceeding as above we get
$$
V_\rho(\{T_{t,1,1}\}_{t>0})(f)(x)\leq C\int_0^{\frac{x}{2}}\mathbb K_1(x,y) |f(y)|y^{2\lambda}dy,\quad x\in(0,\infty),
$$
where
$$
\mathbb K_1(x,y)=\int_0^\infty |\partial_tK_{t,1}(x,y)|dt,\quad x,y \in (0,\infty).
$$
By using again property (a) for $k=0,1$ we have that
\begin{eqnarray*}
\mathbb K_1(x,y) &\leq& C\int_0^\infty \frac{1}{t^{2\lambda +2}}\int_0^{\frac{\pi}{2}}(|\phi(u)|+|\phi'(u)|u)_{|u=\frac{\sqrt{x^2+y^2-2xy\cos \theta}}{t}}\sin ^{2\lambda -1}\theta d\theta dt\\
&\leq & C \int^{\frac{\pi}{2}}_0 \frac{\sin^{2\lambda-1}\theta}{((x-y)^2+2xy(1-\cos \theta))^{\lambda +\frac{1}{2}}}d\theta \\
&\leq& C\frac{1}{|x-y|^{2\lambda+1}},\quad x,y \in (0,\infty),\;x \not=y.
\end{eqnarray*}
Then,
\begin{equation}\label{(2.5)}
V_\rho(\{T_{t,1,1}\}_{t>0})(f)(x) \leq \frac{C}{x^{2\lambda+1}}\int_0^x|f(y)|y^{2\lambda}dy=CH_0^\lambda (|f|)(x),\quad x\in (0,\infty).
\end{equation}
In a similar way we can see that
\begin{equation}\label{(2.6)}
V_\rho(\{T_{t,1,3}\}_{t>0})(f)(x) \leq C\int_x^\infty |f(y)| \frac{dy}{y}=CH_\infty (|f|)(x),\quad x \in (0,\infty).
\end{equation}
On the other hand, we write
\begin{eqnarray*}
K_{t,1}(x,y) &=& \frac{\Gamma(\lambda+\frac{1}{2})}{\sqrt\pi\Gamma(\lambda)}\frac{1}{t^{2\lambda+1}}\int_0^{\frac{\pi}{2}}\phi\left(\frac{\sqrt{x^2+y^2-2xy\cos\theta}}{t}\right)\sin^{2\lambda -1}\theta d\theta \\
&=& \frac{\Gamma(\lambda+\frac{1}{2})}{\sqrt\pi\Gamma(\lambda)}\frac{1}{t^{2\lambda+1}}\left[\int_0^{\frac{\pi}{2}}\phi\left(\frac{\sqrt{x^2+y^2-2xy\cos\theta}}{t}\right)(\sin^{2\lambda -1}\theta -\theta^{2\lambda -1})d\theta \right.\\
&&+ \int_0^{\frac{\pi}{2}}\theta^{2\lambda-1}\left( \phi\left(\frac{\sqrt{x^2+y^2-2xy\cos\theta}}{t}\right) - \phi \left(\frac{\sqrt{(x-y)^2+xy\theta^2}}{t}\right)\right)d\theta \\
&&-\frac{1}{2}\left(\frac{t}{\sqrt{xy}}\right)^{2\lambda} \int_{\frac{\pi^2xy}{4t^2}}^\infty u^{\lambda -1}\phi \left(\sqrt{\frac{(x-y)^2}{t^2}+u}\right)du\\
&&+\left.\frac{1}{2}\left(\frac{t}{\sqrt{xy}}\right)^{2\lambda}\int_0^\infty u^{\lambda-1}\phi\left(\sqrt{\frac{(x-y)^2}{t^2}+u}\right)du\right] \\
&=&\sum_{j=1}^4K_{t,1}^j(x,y),\quad t,x,y \in (0,\infty).
\end{eqnarray*}
We define, for every $j=1,2,3,4$,
$$
\mathfrak{K}_{t,j}(f)(x) = \int_{\frac{x}{2}}^{2x} K_{t,1}^j(x,y)f(y) y^{2\lambda} dy,\quad t,x\in(0,\infty).
$$
We get
$$
V_\rho(\{T_{t,1,2}\}_{t>0})(f)  \leq \sum_{j=1}^4 V_\rho(\{\mathfrak K_{t,j}\}_{t>0})(f).
$$
As above we can obtain that, for $j=1,2,3$,
$$
V_\rho(\{\mathfrak K_{t,j}\}_{t>0})(f)(x)  \leq C\int_{\frac{x}{2}}^{2x}|f(y)|y^{2\lambda}\int_0^\infty |\partial_tK^j_{t,1}(x,y)|dtdy ,\quad x\in(0,\infty).
$$
Since $|\sin^{2\lambda-1}\theta - \theta^{2\lambda-1}| \leq C\theta^{2\lambda +1}$, $\theta \in (0,\frac{\pi}{2})$, we get
\begin{eqnarray*}
\int_0^\infty |\partial_tK^1_{t,1}(x,y)|dt &\leq & C\int_0^\infty \frac{1}{t^{2\lambda+2}}\int_0^{\frac{\pi}{2}}\theta^{2\lambda +1}
(|\phi(u)|+|\phi'(u|)u)_{|u=\frac{\sqrt{x^2+y^2-2xy\cos\theta}}{t}}d\theta dt\\
&\leq & C\int_0^{\frac{\pi}{2}}\frac{\theta^{2\lambda+1}}{((x-y)^2+2xy(1-\cos \theta))^{\lambda+\frac{1}{2}}}d\theta \leq C\int_0^{\frac{\pi}{2}} \frac{\theta^{2\lambda+1}}{(xy(1-\cos \theta))^{\lambda+\frac{1}{2}}} d\theta \\
&\leq &  \frac{C}{(xy)^{\lambda+\frac{1}{2}}} \leq \frac{C}{x^{2\lambda+1}},\quad 0<\frac{x}{2}<y<2x.
\end{eqnarray*}
We have used that $\phi$ satisfies (a) for $k=0,1$.
Then
\begin{equation}\label{(2.7)}
V_\rho(\{\mathfrak K_{t,1}\}_{t>0})(f)(x) \leq \frac{C}{x^{2\lambda+1}}\int_{\frac{x}{2}}^{2x}|f(y)|y^{2\lambda}dy\,\quad x \in (0,\infty).
\end{equation}
We have that
\begin{align*}
\partial_tK^2_{t,1}(x,y) &= -\frac{\Gamma(\lambda +\frac{1}{2})}{\sqrt\pi \Gamma (\lambda)} \frac{1}{t^{2\lambda +1}}\\
&\hspace{-1cm}\times \left[\frac{2\lambda +1}{t} \int_0^{\frac{\pi}{2}}\theta^{2\lambda -1}\left(\phi\left(\frac{\sqrt{(x-y)^2+2xy(1-\cos\theta)}}{t}\right)\right. -\phi\left(\frac{\sqrt{(x-y)^2+xy\theta^2}}{t}\right)\right)d\theta\\
&\hspace{-1cm}+\int_0^{\frac{\pi}{2}}\theta^{2\lambda -1}\left(\phi'\left(\frac{\sqrt{(x-y)^2+2xy(1-\cos\theta)}}{t}\right)\frac{\sqrt{(x-y)^2+2xy(1-\cos\theta)}}{t^2}\right.\\
&\hspace{-1cm} \left.\left.-\phi'\left(\frac{\sqrt{(x-y)^2+xy\theta^2}}{t}\right)\frac{\sqrt{(x-y)^2+xy\theta^2}}{t^2}\right)d\theta\right]\\
&\hspace{-1cm}=K_{t,1}^{2,1}(x,y) + K_{t,1}^{2,2}(x,y),\quad t,x,y \in (0,\infty).
\end{align*}
We can write
$$
K^{2,1}_{t,1} (x,y)= -\frac{\Gamma(\lambda + \frac{3}{2})}{\sqrt\pi\Gamma(\lambda)t^{2\lambda+2}}\int_0^{\frac{\pi}{2}}\theta^{2\lambda-1}
\int_{\theta^2}^{2(1-\cos\theta)}\phi'\left(\frac{\sqrt{(x-y)^2+xyz}}{t}\right)\frac{xy}{t\sqrt{(x-y)^2+xyz}}dz d\theta.
$$
Then, since $\phi$ satisfies (a) for $k=1$, we get
\begin{align*}
\int_0^\infty \left|K^{2,1}_{t,1}(x,y)\right|dt \\
&\hspace{-2cm}\leq C \int_0^\infty\frac{1}{t^{2\lambda+3}}\int_0^{\frac{\pi}{2}}\theta^{2\lambda-1}\int_{2(1-\cos\theta)}^{\theta^2}
\left|\phi'\left(\frac{\sqrt{(x-y)^2+xyz}}{t}\right)\right|\frac{xy}{\sqrt{(x-y)^2+xyz}}dz d\theta dt\\
&\hspace{-2cm}\leq Cxy\int_0^{\frac{\pi}{2}}\int_{2(1-\cos\theta)}^{\theta^2}\frac{\theta^{2\lambda-1}}{((x-y)^2+xyz)^{\lambda+\frac{3}{2}}}\int_0^\infty
u^{2\lambda+1}|\phi'(u)|dudzd\theta\\
&\hspace{-2cm}\leq Cxy\int_0^{\frac{\pi}{2}}\frac{\theta^{2\lambda-1}(\theta^2-2(1-\cos\theta))}{((x-y)^2+xy\theta^2)^{\lambda+\frac{3}{2}}}d\theta\leq Cxy\int_0^{\frac{\pi}{2}}\frac{\theta^{2\lambda+3}}{((x-y)^2+xy\theta^2)^{\lambda+\frac{3}{2}}}d\theta\\
&\hspace{-2cm}\leq \frac{C}{(xy)^{\lambda+\frac{1}{2}}} \leq \frac{C}{x^{2\lambda+1}}, \quad 0<\frac{x}{2}<y<2x.
\end{align*}
We have that
\begin{eqnarray*}
K_{t,1}^{2,2}(x,y) &=& -\frac{\Gamma(\lambda+\frac{1}{2})}{2\sqrt{\pi}\Gamma(\lambda)t^{2\lambda-1}}\int_0^{\frac{\pi}{2}}\theta^{2\lambda -1}\int_{\theta^2}^{2(1-\cos\theta)}\left[\phi''\left(\frac{\sqrt{(x-y)^2+xyz}}{t}\right)\frac{xy}{t^3}\right.\\
&&\left.+\phi'\left(\frac{\sqrt{(x-y)^2+xyz}}{t}\right)\right]\frac{xy}{\sqrt{(x-y)^2+xyz}t^2}dzd\theta,\quad t,x,y \in(0,\infty).
\end{eqnarray*}
Since $\phi$ satisfies (a) for $k=1,2$ we get
\begin{eqnarray*}
\int_0^\infty |K_{t,1}^{2,2}(x,y)|dt &\leq & Cxy\int_0^\infty\frac{1}{t^{2\lambda+1}}\int_0^{\frac{\pi}{2}}\theta^{2\lambda-1} \int_{2(1-\cos\theta)}^{\theta^2}\left[\left|\phi''\left(\frac{\sqrt{(x-y)^2+xyz}}{t}\right)\right|\frac{1}{t^3}\right.\\
&& \left. +\left|\phi'\left(\frac{\sqrt{(x-y)^2+xyz}}{t}\right)\right|\frac{1}{t^2\sqrt{(x-y)^2+xyz}}\right]dzd\theta dt\\
&\leq &Cxy\int_0^{\frac{\pi}{2}}\theta^{2\lambda-1}\int_{2(1-\cos\theta)}^{\theta^2}\left[\int_0^\infty |\phi''(u)|\left(\frac{u}{\sqrt{(x-y)^2+xyz}}\right)^{2\lambda+4}\right.\\
&&\left.\times \frac{\sqrt{(x-y)^2+xyz}}{u^2}du+\int_0^\infty |\phi'(u)|\left(\frac{u}{\sqrt{(x-y)^2+xyz}}\right)^{2\lambda+3}\frac{du}{u^2}\right]dzd\theta\\
&\leq& Cxy \int_0^{\frac{\pi}{2}}\theta^{2\lambda-1}\int_{2(1-\cos\theta)}^{\theta^2}\frac{dzd\theta}{((x-y)^2+xyz)^{\lambda+\frac{3}{2}}}
\left(\int_0^\infty|\phi''(u)|u^{2\lambda+2}du\right.\\
&&\left.+\int_0^\infty|\phi'(u)|u^{2\lambda+1}du \right)\\
&\leq& \frac{C}{x^{2\lambda+1}}, \quad 0 < \frac{x}{2}<y<2x.
\end{eqnarray*}
We conclude that
\begin{equation}\label{(2.8)}
V_\rho(\{\mathfrak K_{t,2}\}_{t>0})(f)(x) \leq \frac{C}{x^{2\lambda+1}}\int_{\frac{x}{2}}^{2x}|f(y)|y^{2\lambda}dy,\quad x\in(0,\infty).
\end{equation}
We have that
\begin{eqnarray*}
\partial_tK^3_{t,1}(x,y) &=&\frac{\Gamma(\lambda+\frac{1}{2})}{2\sqrt{\pi}\Gamma(\lambda)(xy)^\lambda}\left[\frac{1}{t^2}\int_{\frac{\pi^2xy}{4t^2}}^\infty u^{\lambda-1}\phi\left(\sqrt{\frac{(x-y)^2}{t^2}+u}\right)du\right.\\
&&-\frac{\pi^2xy}{2t^4}\left(\frac{\pi^2xy}{4t^2}\right)^{\lambda-1}\phi\left(\sqrt{\frac{(x-y)^2}{t^2}+\frac{\pi^2xy}{4t^2}}\right)\\
&&\left.-\frac{1}{t}\int_{\frac{\pi^2xy}{4t^2}}^\infty u^{\lambda-1}\phi'\left(\sqrt{\frac{(x-y)^2}{t^2}+u}\right)\frac{(x-y)^2}{t^3\sqrt{\frac{(x-y)^2}{t^2}+u}}du\right]\\
&=& \sum_{j=1}^3 K^{3,j}_{t,1}(x,y),\quad t,x,y\in (0,\infty).
\end{eqnarray*}
It follows that
\begin{eqnarray}\label{(2.9)}
\int_0^\infty |K^{3,1}_{t,1}(x,y)|dt &\leq& \frac{C}{(xy)^{\lambda}}\int_0^\infty \frac{1}{t^2}\int_{\frac{\pi^2xy}{4t^2}}^\infty u^{\lambda-1}\left|\phi\left(\sqrt{\frac{(x-y)^2}{t^2} +u}\right)\right|dudt\nonumber\\
&\leq& \frac{C}{|x-y|(xy)^\lambda}\int_0^\infty\frac{1}{\sqrt{s}}\int_{\frac{\pi^2xys}{4|x-y|^2}}^\infty  u^{\lambda-1}|\phi(\sqrt{s+u})|duds\nonumber\\ 
&=:&H(x,y),\quad x,y\in (0,\infty),\,x\not=y.
\end{eqnarray}
Since $\phi$ satisfies (b) for $k=0$ we obtain, for every $x\in(0,\infty)$,
\begin{equation}\label{(2.10)}
\int_{\frac{x}{2}}^{2x}H(x,y) y^{2\lambda}dy=C\int_{\frac{1}{2}}^2\frac{v^\lambda }{|1-v|}\int_0^\infty \frac{1}{\sqrt{s}}\int_{\frac{\pi^2vs}{4(1-v)^2}}^\infty u^{\lambda-1}|\phi(\sqrt{s+u})|dudsdv < \infty.
\end{equation}
By using Jensen's inequality, we get from (\ref{(2.10)}) that the operator $T_\lambda$ defined by
$$
T_\lambda(f)(x) = \int_{\frac{x}{2}}^{2x} H(x,y)f(y)y^{2\lambda}dy,\quad x\in (0,\infty),
$$
is bounded from $L^p((0,\infty),m_\lambda)$ into itself for every $1 \leq p<\infty$.

We also have that
\begin{align}\label{(2.11)}
\int_0^\infty |K^{3,2}_{t,1}(x,y)|dt &\leq C\int_0^\infty \frac{1}{t^{2\lambda+2}}\phi\left(\sqrt{\frac{(x-y)^2+\pi^2xy}{t^2}}\right)dt\nonumber\\
&\leq \frac{C}{((x-y)^2+\pi^2xy)^{\lambda+\frac{1}{2}}}\int_0^\infty u^{2\lambda}|\phi(u)|du \leq \frac{C}{x^{2\lambda+1}},\quad 0 < \frac{x}{2} <y < 2x.
\end{align}
By proceeding as in the $K^{3,1}_{t,1}$ case we deduce that
\begin{align}\label{(2.12)}
\int_0^\infty|K^{3,3}_{t,1}(x,y)| dt \leq & C\frac{|x-y|^2}{(xy)^\lambda}\int_0^\infty\frac{1}{t^4}\int_{\frac{\pi^2xy}{4t^2}}^\infty u^{\lambda-1}\left|\phi'\left(\sqrt{\frac{(x-y)^2}{t^2}+u}\right)\right|\frac{1}{\sqrt{\frac{(x-y)^2}{t^2}+u}}dudt\nonumber\\
&\leq \frac{C}{|x-y|(xy)^{\lambda}}\int_0^\infty \sqrt{s}\int_{\frac{\pi^2xys}{4|x-y|^2}}^\infty u^{\lambda-1}|\phi'(\sqrt{s+u})|\frac{du}{\sqrt{s+u}}ds\nonumber\\
&=:F(x,y),\quad x,y\in(0,\infty),\,x\not=y.
\end{align}

Since $\phi$ satisfies (b) for $k=1$ we obtain, for every $x \in (0,\infty)$,
\begin{equation}\label{(2.13)}
\int_{\frac{x}{2}}^{2x}F(x,y)y^{2\lambda}dy = C\int_{\frac{1}{2}}^2\frac{v^\lambda}{|1-v|}\int_0^\infty \sqrt{s}\int_{\frac{\pi^2vs}{4(1-v)^2}}^\infty u^{\lambda-1}\frac{|\phi'(\sqrt{u+s})|}{\sqrt{u+s}}dudsdv < \infty.
\end{equation}
Again by Jensen's inequality, (\ref{(2.13)}) implies that the operator $\mathcal{S}_\lambda$ defined by
$$
\mathcal S_\lambda(f)(x) = \int_{\frac{x}{2}}^{2x}F(x,y) f(y) y^{2\lambda}dy,\quad x \in (0,\infty),
$$
is bounded from $L^p((0,\infty),m_\lambda)$ into itself for every $1 \leq p < \infty$.
By using (\ref{(2.9)}), (\ref{(2.11)}) and (\ref{(2.12)}) we conclude that
\begin{equation}\label{(2.14)}
V_\rho(\{\mathfrak{K}_{t,3}\}_{t>0})(f)(x) \leq C\left(T_\lambda(|f|)(x) + \mathcal S_\lambda(|f|)(x) + \frac{1}{x^{2\lambda +1}}\int_{\frac{x}{2}}^{2x}|f(y)|y^{2\lambda}dy\right), \quad x\in (0,\infty),
\end{equation}
From (\ref{(2.7)}), (\ref{(2.8)}) and (\ref{(2.14)}) it follows that
\begin{equation}\label{(2.15)}
\sum^3_{j=1}V_\rho(\{\mathfrak{K}_{t,j}\}_{t>0})(f)(x) \leq C\left(T_\lambda(|f|)(x) + \mathcal S_\lambda(|f|)(x) + \frac{1}{x^{2\lambda +1}}\int_{\frac{x}{2}}^{2x}|f(y)|y^{2\lambda}dy\right), \quad x\in (0,\infty).
\end{equation}
Since $\displaystyle\frac{1}{x^{2\lambda+1}}\int_{\frac{x}{2}}^{2x}y^{2\lambda}dy = \int_{\frac{1}{2}}^2v^{2\lambda}dv$, $x\in (0,\infty)$, Jensen's inequality implies that the operator $\mathfrak{D}_\lambda$ defined by
$$
\mathfrak{D}_\lambda(f)(x)=\frac{1}{x^{2\lambda+1}}\int_{\frac{x}{2}}^{2x}f(y) y^{2\lambda}dy,\quad x\in (0,\infty),
$$
is bounded from $L^p((0,\infty ),m_\lambda )$ into itself for $1\leq p<\infty$.

By using (\ref{(2.5)}), (\ref{(2.6)}) and (\ref{(2.15)}) we can see that the operator
$$
V_\rho(\{T_{t,1,1}\}_{t>0}) + V_\rho(\{T_{t,1,3}\}_{t>0})+\sum_{j=1}^3V_\rho(\{\mathfrak K_{t,j}\}_{t>0})
$$
is bounded from $L^p((0,\infty),m_\lambda)$ into itself, for every $1<p<\infty$, and from $L^1((0,\infty),m_\lambda)$
into $L^{1,\infty}((0,\infty),m_\lambda)$.

We now study the operator $V_\rho(\{\mathfrak K_{t,4}\}_{t>0})$. We recall that
$$
\mathfrak K_{t,4}(f)(x)=\int_{\frac{x}{2}}^{2x}K^4_{t,1}(x,y)f(y)y^{2\lambda}dy,\quad x\in (0,\infty),
$$
where
$$
K^4_{t,1}(x,y) = \frac{\Gamma(\lambda+\frac{1}{2})}{2\sqrt{\pi}\Gamma(\lambda)}\frac{1}{t(xy)^\lambda}\int_0^\infty u^{\lambda -1}\phi\left(\sqrt{\left(\frac{x-y}{t}\right)^2 +u}\right)du,\quad t,x,y \in (0,\infty).
$$
We consider the even function
$$
\Phi(z)=\int_0^\infty u^{\lambda-1}\phi(\sqrt{z^2+u})du,\quad z \in \mathbb R.
$$
We define, for every $t>0$, the classical convolution operator $\mathcal S_t$ by
$$
\mathcal S_t(f) = f\ast\Phi_{(t)},
$$
where $\Phi_{(t)}(x) =  \frac{1}{t}\Phi\left(\frac{x}{t}\right)$, $x \in \mathbb R$. By taking into account properties (c) and (d), according to \cite[Lemma 2.4]{CJRW1} the $\rho$-variation operator $V_\rho(\{\mathcal S_t\}_{t>0})$ is bounded from $L^p(\mathbb R,dx)$ into itself, for every $1<p<\infty$, and from $L^1(\mathbb R,dx)$ into $L^{1,\infty}( \mathbb R,dx)$.

We can write
$$
K^4_{t,1}(x,y)= \frac{\Gamma(\lambda+\frac{1}{2})}{2\sqrt\pi\Gamma(\lambda)}\frac{1}{(xy)^\lambda}\Phi_{(t)}(x-y),\quad t,x,y \in (0,\infty),
$$
and
$$
\mathfrak K_{t,4}(f)(x)=\frac{\Gamma(\lambda+\frac{1}{2})}{2\sqrt\pi\Gamma(\lambda)}x^{-\lambda}\mathcal S_t\left(y^{\lambda} f(y)\chi_{(\frac{x}{2},2x)}(y)\right)(x),\quad t,x\in(0,\infty).
$$
Let $j \in \mathbb Z$. We have that
$$
\mathfrak K_{t,4}(f)(x)= \left(\int_{2^{j-1}}^{2^{j+2}}-\int_{2^{j-1}}^{\frac{x}{2}}- \int_{2x}^{2^{j+2}}\right) K^4_{t,1}(x,y)f(y)y^{2\lambda}dy,\quad x \in[2^j,2^{j+1}),\;t\in (0,\infty ).
$$
Then,
\begin{eqnarray*}
\mathfrak K_{t,4}(f)(x) &=& \sum_{j  \in  \mathbb Z}\chi_{[2^j,2^{j+1})}(x)\frac{\Gamma(\lambda+\frac{1}{2})}{2\sqrt\pi\Gamma(\lambda)}x^{-\lambda}\mathcal S_t(y^\lambda f(y)\chi_{[2^{j-1},2^{j+2})}(y))(x)\\
&&-\sum_{j \in\mathbb Z}\chi_{[2^j,2^{j+1})}(x) \frac{\Gamma(\lambda+\frac{1}{2})}{2\sqrt\pi\Gamma(\lambda)}\int_{2^{j-1}}^{\frac{x}{2}}(xy)^{-\lambda}\frac{1}{t}\Phi\left(\frac{x-y}{t}\right)f(y) y^{2\lambda}dy\\
&&-\sum_{j \in\mathbb Z}\chi_{[2^j,2^{j+1})}(x) \frac{\Gamma(\lambda+\frac{1}{2})}{2\sqrt\pi\Gamma(\lambda)}\int_{2x}^{2^{j+2}}(xy)^{-\lambda}\frac{1}{t}\Phi\left(\frac{x-y}{t}\right)f(y) y^{2\lambda}dy\\
&=& \sum_{i=1}^3 L_{t,4}^i(f)(x),\quad t,x \in (0,\infty).
\end{eqnarray*}
Let $1<p<\infty$. We have that
\begin{eqnarray*}
\int_0^\infty |V_\rho(\{L^1_{t,4}\}_{t>0})(f)(x)|^p x^{2\lambda}dx &&\\
& &\hspace{-5cm}\leq C\sum_{j\in \mathbb Z}\int_{2^j}^{2^{j+1}}x^{(2-p)\lambda}|V_\rho(\{\mathcal S_t\}_{t>0})(y^\lambda f(y)\chi_{[2^{j-1},2^{j+2})}(y))(x)|^pdx\\
& &\hspace{-5cm}\leq C\sum_{j\in\mathbb Z}2^{j(2-p)\lambda}\int_0^\infty |V_\rho(\{\mathcal S_t\}_{t>0})(y^\lambda f(y)\chi_{[2^{j-1},2^{j+2})}(y))(x)|^pdx\\
& &\hspace{-5cm}\leq C\sum_{j\in\mathbb Z}2^{j(2-p)\lambda}\int_{2^{j-1}}^{2^{j+2}}|y^\lambda f(y)|^pdy\leq C\sum_{j \in \mathbb Z} \int_{2^{j-1}}^{2^{j+2}} |f(y)|^py^{2\lambda}dy \\
& &\hspace{-5cm} \leq C\int_0^\infty |f(y)|^py^{2\lambda}dy,\quad f \in L^p((0,\infty),m_\lambda).
\end{eqnarray*}
We also get, for every $\alpha >0$,
\begin{align*}
m_\lambda\Big(\big\{x\in(0,\infty): V_\rho(\{L^1_{t,4}\}_{t>0})(f)(x)>\alpha\big\}\Big)&\\
& \hspace{-5.5cm}\leq \sum_{j\in Z}m_\lambda\Big(\big\{x\in[2^j,2^{j+1}):\frac{\Gamma(\lambda+\frac{1}{2})}{2\sqrt\pi\Gamma(\lambda)}x^{-\lambda} V_\rho(\{\mathcal S_t\}_{t>0})(y^\lambda f(y)\chi_{[2^{j-1},2^{j+2})}(y))(x) >\alpha\big\}\Big)\\
&  \hspace{-5.5cm} \leq C \sum_{j\in \mathbb Z}2^{2j\lambda}\left|\big\{x\in(0,\infty): V_\rho(\{\mathcal S_t\}_{t>0})(y^\lambda f(y) \chi_{[2^{j-1},2^{j+2})}(y))(x) >C 2^{\lambda j}\alpha\big\}\right|\\
& \hspace{-5.5cm} \leq \frac{C}{\alpha}\sum_{j \in \mathbb Z}2^{\lambda j}\int_{2^{j-1}}^{2^{j+2}}y^\lambda|f(y)|dy\leq \frac{C}{\alpha}\int_0^\infty |f(y)|y^{2\lambda}dy,\quad f \in L^1((0,\infty),m_\lambda).
\end{align*}
On the other hand, by using property (e) we can write
\begin{eqnarray*}
V_\rho(\{L^2_{t,4}\}_{t>0})(f)(x) &\leq & C \sum_{j \in \mathbb Z}\chi_{[2^j,2^{j+1})}(x) \int_{2^{j-1}}^{\frac{x}{2}}\left(\frac{y}{x}\right)^\lambda|f(y)|\int_0^\infty\left|\partial_t\left(\frac{1}{t}\Phi\left(\frac{x-y}{t}\right)\right)\right|dtdy\\
& & \hspace{-3cm}\leq C\sum_{j\in \mathbb Z}\chi_{[2^j,2^{j+1})}(x)\int_{2^{j-1}}^{\frac{x}{2}}\left(\frac{y}{x}\right)^\lambda|f(y)|\int_0^\infty \left(\frac{1}{t^2}\left|\Phi\left(\frac{x-y}{t}\right)\right|+ \frac{|x-y|}{t^3}\left|\Phi'\left(\frac{x-y}{t}\right)\right|\right)dtdy\\
& & \hspace{-3cm} \leq C\sum_{j\in \mathbb Z}\chi_{[2^j,2^{j+1})}(x)\int_{2^{j-1}}^{\frac{x}{2}}\left(\frac{y}{x}\right)^\lambda\frac{|f(y)|}{|x-y|}dy\int_0^\infty (|\Phi(u)|+u|\Phi'(u)|)du\\
& & \hspace{-3cm} \leq C\sum_{j\in \mathbb Z}\chi_{[2^j,2^{j+1})}(x)\frac{1}{x^{\lambda+1}}\int_{2^{j-1}}^{\frac{x}{2}}\frac{y^{2\lambda}}{2^{\lambda j}}|f(y)|dy\leq C\sum_{j\in \mathbb Z}\chi_{[2^j,2^{j+1})}(x)\frac{1}{x^{2\lambda+1}}\int_{2^{j-1}}^{\frac{x}{2}}y^{2\lambda}|f(y)|dy\\
& & \hspace{-3cm} \leq \frac{C}{x^{2\lambda+1}}\int_0^{\frac{x}{2}}y^{2\lambda}|f(y)|dy\leq CH_0^\lambda (|f|)(x),\quad x\in(0,\infty).
\end{eqnarray*}
In a similar way we get
$$
V_\rho(\{L^3_{t,4}\}_{t>0})(f)(x) \leq C\int_{2x}^\infty \frac{|f(y)|}{y}dy\leq CH_\infty (|f|)(x), \quad x\in(0,\infty).
$$
By using again the $L^p$-boundedness properties of Hardy type operators \cite[Lemmas 2.6 and 2,7]{BHNV} we conclude that the operators $V_\rho( \{L^j_{t,4}\}_{t>0})$, $j=2,3$, are bounded from $L^p((0,\infty),m_\lambda)$ into itself, for every $1<p<\infty$, and from $L^1((0,\infty),m_\lambda)$ into $L^{1,\infty}((0,\infty),m_\lambda)$.

It follows that the operator $V_\rho(\{\mathfrak K_{t,4}\}_{t>0})$ is bounded from $L^p((0,\infty),m_\lambda$) into itself, for every $1<p<\infty$ and from $L^1((0,\infty),m_\lambda)$ into $L^{1,\infty}((0,\infty),m_\lambda)$.
By putting together all the properties we have just proved we establish that $V_\rho(\{\phi_t\}_{t>0})$ is bounded from $L^p((0,\infty),m_\lambda)$ into itself, for every $1<p<\infty$, and from $L^1((0,\infty),m_\lambda)$ into $L^{1,\infty}((0,\infty),m_\lambda)$.
\end{proof}

We now consider the operator
\begin{align*}
V_\rho ^{\lambda ,\#}(\{\phi _t\}_{t>0})(f)(x)&\\
&\hspace{-2cm}=\sup_{\substack{Q\,{\rm cube}\\Q\ni x}}\esssup_{y_1,y_2\in Q}\big|V_\rho ^\lambda (\{\phi _t\}_{t>0})(f\mathcal{X}_{(3Q)^c})(y_1)-V_\rho ^\lambda (\{\phi _t\}_{t>0})(f\mathcal{X}_{(3Q)^c})(y_2)\big|,\quad x\in (0,\infty).
\end{align*}

\begin{prop}\label{Prop2.4}
Let $\lambda >-1/2$ and $\rho >2$. Assume that $\phi $ is a twice differentiable function on $(0,\infty )$ such that there exists $C>0$ for which $|\Phi'(x)|\leq C(1+x^2)^{-(\lambda +3/2)}$, $x\in (0,\infty)$, where $\Phi (x)=(2\lambda +1)\phi (x)+x\phi '(x)$, $x\in (0,\infty )$.

Then, the operator $V_\rho ^{\lambda ,\#}$ is bounded from $L^p((0,\infty ),m_\lambda )$ into itself, for every $1<p<\infty$ and from $L^1((0,\infty ),m_\lambda )$ into $L^{1,\infty}((0,\infty ),m_\lambda )$.
\end{prop}

\begin{proof}
Let $x\in (0,\infty )$. Suppose that $Q$ is an interval in $(0,\infty )$ such that $x\in Q$.

By using Minkowski inequality we obtain
\begin{align*}
    \big|V_\rho ^\lambda(\{\phi _t\}_{t>0})(f\mathcal{X}_{(3Q)^c})(y_1)-V_\rho ^\lambda (\{\phi _t\}_{t>0})(f\mathcal{X}_{(3Q)^c})(y_2)\big|&\\
    &\hspace{-6cm}\leq \int_{(0,\infty)\setminus (3Q)}|f(z)|z^{2\lambda } \|\{\,_\lambda \tau_{y_1}(\phi _t)(z)-\,_\lambda\tau_{y_2}(\phi _t)(z)\}_{t>0}\|_{v_\rho}dz,\quad y_1,y_2\in Q.
\end{align*}
We get as above that
$$
\|\{\,_\lambda\tau_{y_1}(\phi _t)(z)-\,_\lambda\tau_{y_2}(\phi _t)(z)\}_{t>0}\|_{v_\rho} \leq \int_0^\infty \big|\partial _t(\,_\lambda \tau_{y_1}(\phi _t)(z)-\,_\lambda\tau_{y_2}(\phi _t)(z))\big|dt.
$$
We can write 
\begin{align*}
\partial _t\,_\lambda\tau_{y}(\phi _t)(z)&=\frac{\Gamma (\lambda +1/2)}{\sqrt{\pi}\Gamma (\lambda)}\partial _t\left[\frac{1}{t^{2\lambda +1}}\int_0^\pi \phi \Big(\frac{\sqrt{z^2+y^2-2zy\cos \theta}}{t}\Big)\sin ^{2\lambda -1}\theta d\theta\right]\\
&=-\frac{\Gamma (\lambda +1/2)}{\sqrt{\pi}\Gamma (\lambda )}\left(\frac{2\lambda+1}{t^{2\lambda +2}}\int_0^\pi \phi \Big(\frac{\sqrt{z^2+y^2-2zy\cos \theta}}{t}\Big)\sin ^{2\lambda -1}\theta d\theta \right.\\
&\quad \left. +\frac{1}{t^{2\lambda +1}}\int_0^\pi \frac{\sqrt{z^2+y^2-2zy\cos \theta}}{t^2}\phi '\Big(\frac{\sqrt{z^2+y^2-2zy\cos \theta}}{t}\Big)\sin ^{2\lambda -1}\theta d\theta\right)\\
&=-\frac{\Gamma (\lambda +1/2)}{\sqrt{\pi}\Gamma (\lambda )}\frac{1}{t^{2\lambda +2}}\int_0^\pi \Phi \Big(\frac{\sqrt{z^2+y^2-2zy\cos \theta}}{t}\Big)\sin ^{2\lambda -1}\theta d\theta,\quad z,y,t\in(0,\infty),
\end{align*}
where $\Phi (z)=(2\lambda +1)\phi (z)+z\phi'(z)$, $z\in (0,\infty)$. Then,
$$
\partial _t\,_\lambda\tau_{y}(\phi _t)(z)=-\frac{1}{t}\,_\lambda\tau_{y}(\Phi _t)(z),\quad z,y,t\in (0,\infty).
$$
By proceeding as in the proof of \cite[(2.8)]{YY}, since $|\Phi'(x)|\leq C(1+x^2)^{-\lambda -3/2}$, $x\in (0,\infty )$, we can see that
\begin{align*}
    \big|\partial_t(\,_\lambda\tau_{y_1}(\phi _t)(z)-\,_\lambda\tau_{y_2}(\phi _t)(z))\big|&= \frac{1}{t}\big|\,_\lambda\tau_{y_1}(\Phi _t)(z)-\,_\lambda\tau_{y_2}(\Phi _t)(z)\big|\\
    &\hspace{-4cm}\leq \frac{C}{m_\lambda (B(y_1, t))+m_\lambda (B(y_1,|y_1-z|))}\frac{|y_1-y_2|}{(t+|y_1-z|)^2},\quad z\in (0,\infty )\setminus (3Q),\;y_1,y_2\in Q,\;t>0.
\end{align*}
We obtain
$$
\int_0^\infty \big|\partial _t(\,_\lambda \tau _{y_1}(\phi _t)(z)-\,_\lambda \tau _{y_2}(\phi _t)(z))\big|dt\leq C\frac{|y_1-y_2|}{m_\lambda (B(y_1,|y_1-z|))|y_1-z|},\quad z\in (0,\infty )\setminus (3Q),\;y_1,y_2\in Q.
$$
It follows that
\begin{align*}
    \int_{(0,\infty )\setminus (3Q)} \|\{\,_\lambda\tau_{y_1}(\phi _t)(z)-\,_\lambda\tau_{y_2}(\phi _t)(z)\}_{t>0}\|_{v_\rho}|f(z)|z^{2\lambda }dz&\\
    &\hspace{-6cm}\leq C|y_1-y_2|\int_{(0,\infty )\setminus (3Q)}\frac{|f(z)|z^{2\lambda }}{m_\lambda (B(y_1,|y_1-z|)|y_1-z|}dz\\
    &\hspace{-6cm}\leq C|y_1-y_2|\sum_{k=1}^\infty \int_{2^{k+1}Q\setminus 2^kQ}\frac{|f(z)|z^{2\lambda }}{m_\lambda (2^kQ)2^k|Q|}dz\\
    &\hspace{-6cm}\leq C|y_1-y_2|\mathcal{M}_\lambda (f)(x)\sum_{k=1}^\infty \frac{1}{2^k|Q|}\leq C\mathcal{M}_\lambda (f)(x),\quad y_1,y_2\in Q.
\end{align*}
Then,
$$
V_\rho ^{\lambda ,\#}(\{\phi _t\}_{t>0})(f)(x)\leq C\mathcal{M}_\lambda (f)(x).
$$
The $L^p$-boundedness properties of $V_\rho ^{\lambda ,\#}(\{\phi _t\}_{t>0})$ follow from the corresponding $L^p$-boundedness properties of the Hardy-Littlewood maximal function $\mathcal{M}_\lambda$.
\end{proof}
The proof of the Theorem \ref{Th1.1} for $V_\rho ^\lambda (\{\phi _t\}_{t>0})$ can be finished as the proof of Theorem \ref{Th1.1} for $g_\phi ^\lambda $.

\subsection{Applications of Theorem \ref{Th1.1}.}\label{S2.4}
Note that if $\phi\in C^2 (0,\infty )$ is such that $|\phi (x)|\leq C(1+x)^{-2\lambda -2}$, $|\phi '(x)|\leq C(1+x)^{-2\lambda -3}$ and $|\phi ''(x)|\leq C(1+x)^{-2\lambda -4}$, $x\in (0,\infty)$, then $\phi$ satisfies all the conditions $(a)-(h)$ in Proposition \ref{propo 2.3}.

We now present some special functions $\phi$ satisfying the properties in the propositions in this section.

\noindent {\bf (I)} The Poisson semigroup $\{P_t^\lambda \}_{t>0}$ associated with the operator $\Delta _\lambda$ is given by
$$
P_t^\lambda (f)=P _t^\lambda \# _\lambda f,\quad t>0,
$$
where $P^\lambda (x)=\frac{2\lambda \Gamma (\lambda )}{\Gamma (\lambda +1/2)\sqrt{\pi}}(1+x^2)^{-\lambda -1}$, $x\in (0,\infty )$.

It is clear that
$$
\left|\frac{\partial ^k}{\partial x^k}P^\lambda (x)\right|\leq \frac{C}{(1+x)^{2\lambda +2+k}},\quad x\in (0,\infty ),\;k\in \mathbb{N}.
$$
We have that
$$
t\partial _tP_t^\lambda (x)=\frac{1}{t^{2\lambda +1}}\left(-(2\lambda +1)P^\lambda \Big(\frac{x}{t}\Big)-\frac{x}{t}(P ^\lambda )'\Big(\frac{x}{t}\Big)\right)=\psi _t^{\lambda ,1}(x),\quad t,x\in (0,\infty ),
$$
where $\psi ^{\lambda,1} (x)=-(2\lambda +1)P^\lambda (x)-x(P^\lambda )'(x)$, $x\in (0,\infty )$. Then,
$$
\left|\frac{\partial ^k}{\partial x^k}\psi ^{\lambda,1} (x)\right|\leq \frac{C}{(1+x)^{2\lambda +2+k}},\quad x\in (0,\infty ),\;k\in \mathbb{N}.
$$
By proceeding inductively we conclude that, for every $m\in \mathbb{N}$,
$$
t^m\partial_t^mP_t^\lambda (x)=\psi _t^{\lambda ,m}(x),\quad t,x\in (0,\infty ),
$$
where, for certain $c_k^{\lambda ,m}\in \mathbb{R}$, $k=0,...,m$,
$$
\psi ^{\lambda ,m}(x)=\sum_{k=0}^mc_k^{\lambda ,m}x^k(P^\lambda )^{(k)}(x),\quad x\in (0,\infty ).
$$
We have that
$$
\left|\frac{\partial ^r}{\partial x^r}\psi ^{\lambda,m} (x)\right|\leq \frac{C}{(1+x)^{2\lambda +2+r}},\quad x\in (0,\infty ),\;r\in \mathbb{N}.
$$
Let $m\in \mathbb{N}$. According to \cite[(19) p. 24]{EMOT}
\begin{align*}
h_\lambda (\psi ^{\lambda,m})(x)&=h_\lambda (t^m\partial _t^m(P_t^\lambda )_{|t=1})(x)=t^m\partial _t^mh_\lambda (P _t^\lambda )(x)_{|t=1}\\&=\frac{2^{1/2-\lambda}}{\Gamma (\lambda +1/2)}t^m\partial _t^m(e^{-tx})_{|t=1}=p(x)e^{-x},\quad x\in (0,\infty ),
\end{align*}
where $p$ is a polynomial and $p(0)=0$ provided that $m\geq 1$. Hence
$$
\int_0^\infty |h_\lambda (\psi ^{\lambda ,m})(x)|^2\frac{dx}{x}<\infty,\quad \mbox{ when }m\geq 1.
$$

From Theorem \ref{Th1.1} we deduce the following result.
\begin{cor}\label{Cor2.1}
Let $\lambda >0$, $\rho >2$, $m\in \mathbb{N}$, $1<p<\infty$ and $w\in A_p^\lambda (0,\infty )$.

(a) The maximal operator $P_*^{\lambda ,m}$ defined by
$$
P_*^{\lambda ,m}(f)=\sup_{t>0}|t^m\partial _t^mP_t^\lambda (f)|
$$
is bounded from $L^p((0,\infty ),w,m_\lambda)$ into itself and
$$
\|P_*^{\lambda , m}(f)\|_{L^p((0,\infty ),w,m_\lambda )}\leq C[w]_{A_p^\lambda }^{\max\{1/(p-1), 1\}}\|f\|_{L^p((0,\infty ),w,m_\lambda )},
$$
for every $f\in L^p((0,\infty) ,w,m_\lambda )$.

(b) The Littlewood-Paley function $g_{\{P_t^\lambda \}_{t>0}}^{\lambda ,m}$ defined by
$$
g_{\{P_t^\lambda \}_{t>0}}^{\lambda ,m}(f)(x)=\left(\int_0^\infty \Big|t^m\partial _t^mP_t^\lambda (f)(x)\Big|^2\frac{dt}{t}\right)^{1/2},\quad x\in (0,\infty ),
$$
with $m\geq 1$, is bounded from $L^p((0,\infty) ,w,m_\lambda )$ into itself and
$$
\|g_{\{P_t^\lambda \}_{t>0}}^{\lambda , m}(f)\|_{L^p((0,\infty ),w,m_\lambda )}\leq C[w]_{A_p^\lambda }^{\max\{1/(p-1), 1\}}\|f\|_{L^p((0,\infty ),w,m_\lambda )},
$$
for every $f\in L^p((0,\infty) ,w,m_\lambda )$.

(c) The variation operator $V_\rho (\{t^m\partial _t^mP_t^\lambda \}_{t>0})$ is bounded from $L^p((0,\infty) ,w,m_\lambda )$ into itself and 
$$
\|V_\rho (\{t^m\partial _t^mP_t^\lambda \}_{t>0})(f)\|_{L^p((0,\infty ),w,m_\lambda )}\leq C[w]_{A_p^\lambda }^{\max\{1/(p-1), 1\}}\|f\|_{L^p((0,\infty ),w,m_\lambda )},
$$
for every $f\in L^p((0,\infty) ,w,m_\lambda )$.

Here $C>0$ does not depend on $w$.
\end{cor}

\noindent {\bf (II)} Suppose that $\phi \in S(0,\infty )$, the Schwartz class in $(0,\infty )$. It is clear that, for every $k\in \mathbb{N}$, $|\phi ^{(k)}(x)|\leq C(1+x)^{-2\lambda -2-k}$, $x\in (0,\infty )$. According to \cite[Satz 5 and p. 201]{Al}, $h_\lambda (\phi)\in S(0,\infty)$. Then, by \cite[Lemma 3, p. 203]{Al} there exists $C>0$ such that $|h_\lambda (\phi)(x)|\leq Cx^2$, $x\in (0,1)$, provided that $h_\lambda (\phi)(0)=0$.

\begin{cor}
Let $\lambda >0$, $\rho >2$, $1<p<\infty$ and $w\in A_p^\lambda (0,\infty)$. Then,
$$
\|Tf\|_{L^p((0,\infty ),w,m_\lambda )}\leq C[w]_{A_p^\lambda}^{\max\{1/(p-1),1\}}\|f\|_{L^p((0,\infty ),w,m_\lambda )},
$$
for every $f\in L^p((0,\infty ),w,m_\lambda )$, where $C$ does not depend on $w$, when $T=\phi_*^\lambda $, $T=V_\rho ^\lambda (\{\phi _t\}_{t>0})$, with $\phi \in S(0,\infty )$, and $T=g_\phi ^\lambda $, with $\phi \in S(0,\infty )$ and $h_\lambda (\phi)(0)=0$.
\end{cor}

The heat semigroup generated by $-\Delta _\lambda $ is $\{W_t^\lambda \}_{t>0}$ where
$$
W_t^\lambda (f)=W_{\sqrt{2t}}^\lambda \# _\lambda f,\quad t>0,
$$
being $W^\lambda (x)=\frac{2^{(1-2\lambda )/2}}{\Gamma (\lambda +1/2)}e^{-x^2/2}$, $x\in (0,\infty )$. We have that $W ^\lambda \in S(0,\infty)$ and according to \cite[(10) p. 29]{EMOT}, $h_\lambda (W^\lambda )(x)=\frac{2^{(1-2\lambda )/2}}{\Gamma (\lambda +1/2)}e^{-x^2/2}$, $x\in (0,\infty )$. By proceeding as in the Poisson case we deduce that $h_\lambda ((t^m\partial _t^mW_t^\lambda) _{|t=1})(0)=0$, for every $m\in \mathbb{N}$, $m>0$.

\begin{cor}
Let $\lambda >0$, $\rho >2$, $m\in \mathbb{N}$, $1<p<\infty$ and $w\in A_p^\lambda (0,\infty )$. Assertions $(a)$, $(b)$ and $(c)$ in Corollary \ref{Cor2.1} hold when the Poisson semigroup $\{P_t^\lambda \}_{t>0}$ is replaced by the heat semigroup $\{W_t^\lambda \}_{t>0}$.
\end{cor}

\noindent {\bf (III)} We define, for every $t>0$ the Bochner-Riesz means $B_t^{\lambda ,\alpha}$ associated with the operator $\Delta _\lambda$ by
$$
B_t^{\lambda ,\alpha}(f)=h_\lambda \left( \Big(1-\frac{y^2}{t^2}\Big)_+^\alpha h_\lambda (f)\right),
$$
where $(z)_+=\max\{0,z\}$, $z\in \mathbb{R}$. We can write
$$
B_t^{\lambda ,\alpha}(f)=\phi _{1/t}^{\lambda ,\alpha}\# _\lambda f,\quad t>0,
$$
being $\phi ^{\lambda ,\alpha}(z)=2^\alpha \Gamma (\alpha +1)z^{-\alpha -\lambda -1/2}J_{\alpha +\lambda +1/2}(z)$, $z\in (0,\infty )$. We can see that
$$
\psi _*^\lambda (f)=\sup_{t>0}|B_t^{\lambda ,\alpha}(f)|=:B_*^{\lambda ,\alpha}(f),
$$
and, for every $\rho >2$,
$$
V_\rho (\{\psi _t\}_{t>0})=V_\rho (\{B_t^{\lambda ,\alpha }\}_{t>0}),
$$
where $\psi =\phi ^{\lambda ,\alpha}$.

According to \cite[(5.4.3) and (5.11.6)]{Le} we have that, for every $\nu >-1$,
\begin{equation}\label{2.16}
|J_\nu (x)|\leq C\left\{
\begin{array}{ll}
x^{-1/2},&x\in (1,\infty ),\\[0.2cm]
x^\nu ,&x\in (0,1).
\end{array}
\right.
\end{equation}
Thus, we have that
$$
|\phi ^{\lambda, \alpha }(x)|\leq \frac{C}{(1+x)^{\alpha +\lambda +1}},\quad x\in (0,\infty ).
$$
Hence,
$$
|\phi ^{\lambda, \alpha }(x)|\leq \frac{C}{(1+x)^{2\lambda +2}},\quad x\in (0,\infty ),
$$
provided that $\alpha \geq \lambda +1$.

Since $\frac{d}{dz}(z^{-\nu} J_\nu (z))=-z^{-\nu }J_{\nu +1}(z)$, $z\in (0,\infty )$ and $\nu >-1$ (\cite[(5.3.5)]{Le}), we get
$$
\frac{d}{dx}(\phi ^{\lambda ,\alpha}(x))=2^\alpha \Gamma (\alpha +1)x^{-\alpha -\lambda -1/2}J_{\alpha +\lambda +3/2}(x),\quad x\in (0,\infty ).
$$

By using \eqref{2.16} we obtain
$$
\Big|\frac{d}{dx}\phi ^{\lambda ,\alpha} (x)\Big|\leq C\left\{
\begin{array}{ll}
x,&x\in (0,1)\\[0.2cm]
x^{-\alpha -\lambda -1},&x\in (1,\infty)
\end{array}
\right.\leq C\frac{x}{(1+x)^{\alpha +\lambda +2}}\leq C\frac{x}{(1+x)^{2\lambda +4}},\quad x\in (0,\infty ),
$$
when $\alpha \geq \lambda +2$.
According to \cite[(5.3.5)]{Le} we get
$$
\frac{d^2}{dx^2}\phi ^{\lambda ,\alpha} (x)= 2^\alpha \Gamma (\alpha +1)\left(x^{-\alpha -\lambda -3/2}J_{\alpha +\lambda +3/2}(x)-x^{\-\alpha -\lambda -1/2}J_{\alpha +\lambda +5/2}(x)\right),\quad x\in (0,\infty ).
$$
Then,
$$
\Big|\frac{d^2}{dx^2}\phi ^{\lambda ,\alpha} (x)\Big|\leq \frac{C}{(1+x)^{2\lambda +4}},\quad x\in (0,\infty ),
$$
provided that $\alpha \geq \lambda +3$.

The Stein-square function associated with the operator $\Delta _\lambda$ is defined by
$$
G_\lambda ^\alpha (f)(x)=\left(\int_0^\infty \Big|t\partial _tB_t^{\lambda ,\alpha }(f)(x)\Big|^2\frac{dt}{t}\right)^{1/2},\quad x\in (0,\infty ).
$$
We have that
$$
G_\lambda ^\alpha (f)(x)=\left(\Big|(\Phi _{1/t}\# _\lambda f)(x)\Big|^2\frac{dt}{t}\right)^{1/2}=g_\Phi ^\lambda(f)(x),\quad x\in (0,\infty ),
$$
where $\Phi (z)=\phi ^{\lambda ,\alpha }(z)+z\frac{d}{dz}\phi ^{\lambda ,\alpha }(z)$, $z\in (0,\infty )$.

We get 
$$
\frac{d}{dz}\Phi (z)=2\frac{d}{dz}\phi ^{\lambda ,\alpha}(z)+z\frac{d^2}{dz^2}\phi ^{\lambda ,\alpha }(z),\quad z\in (0,\infty ).
$$
As above it follows that
$$
\Big|\frac{d}{dz}\Phi (z)\Big|\leq \frac{C}{(1+x)^{\lambda +\alpha}}\leq \frac{C}{(1+x)^{2\lambda +3}},\quad x\in (0,\infty ),
$$
when $\alpha \geq \lambda +3$.

On the other hand, since $h_\lambda (\phi ^{\lambda,\alpha})(x)=(1-x^2)_+^\alpha$, $x\in (0,\infty)$, (\cite[(7) p. 48]{EMOT}) we obtain that
$
h_\lambda (\Phi )(x)=2\alpha x^2(1-x^2)_+^{\alpha -1}$, $x\in (0,\infty )$. Then, when $\alpha >1/2$,
$$
\int_0^\infty |h_\lambda (\Phi )(x)|^2\frac{dx}{x}<\infty.
$$

From Theorem \ref{Th1.1} we deduce the following result.
\begin{cor}
Let $\lambda >0$, $\rho >2$, $1<p<\infty$ and $w\in A_p^\lambda (0,\infty )$.

(a) If $\alpha \geq \lambda +1$, the maximal operator 
$$
B_*^{\lambda ,\alpha}(f)=\sup_{t>0}|B_t^{\lambda,\alpha} (f)|
$$
is bounded from $L^p(0,\infty ),w,m_\lambda)$ into itself and
$$
\|B_*^{\lambda , \alpha}(f)\|_{L^p((0,\infty ),w,m_\lambda )}\leq C[w]_{A_p^\lambda }^{\max\{1/(p-1), 1\}}\|f\|_{L^p((0,\infty ),w,m_\lambda )},
$$
for every $f\in L^p((0,\infty) ,w,m_\lambda )$.

(b) If $\alpha \geq \lambda +3$, by denoting $T$ the square function $G_\lambda ^\alpha$ or the $\rho$-variation operator $V_\rho (\{B_t^{\lambda,\alpha} \}_{t>0})$, $T$ is bounded from $L^p((0,\infty) ,w,m_\lambda )$ into itself and
$$
\|T(f)\|_{L^p((0,\infty ),w,m_\lambda )}\leq C[w]_{A_p^\lambda }^{\max\{1/(p-1), 1\}}\|f\|_{L^p((0,\infty ),w,m_\lambda )},
$$
for every $f\in L^p((0,\infty) ,w,m_\lambda )$.

Here $C>0$ does not depend on $w$.
\end{cor}

\section{Proof of Theorem \ref{Th1.2}} In order to prove Theorem \ref{Th1.2} we can adapt the procedure used in the proof of \cite[Theorems 1.3 and 3.7]{DGKLWY} to a vector-valued setting. Actually the arguments in the proof of those theorems can be seen as a version in spaces of homogeneous type of the ones developed in \cite{LO} and \cite{LOR}. In \cite[\S 9]{Lori} Lorist commented that the arguments presented in \cite{LOR} can be extended to the vector-valued framework and to spaces of homogeneous type.

Suppose that the triple $(\Omega,d,\mu)$ is a space of homogeneous type where $d$ is a quasimetric on $\Omega$ and $\mu$ is a positive doubling measure on $\Omega$. Assume that $T$ is a linear and bounded operator from $L^1(\Omega,\mu)$ into $L^{1,\infty}_E(\Omega, \mu)$ where $E$ is a Banach space. Given a ball $B_0$ in $\Omega$ we define the maximal type operator $\mathcal{M}_{T,B_0}$ by
$$
\mathcal{M}_{T,B_0}(f)(x)=\sup_{\substack{B\ni x,B\subset B_0\\B \mbox{ ball }}}\|T(f\mathcal{X}_{3B_0\setminus 3B})\|_{L_E^\infty (B)},\quad x\in \Omega.
$$
By proceeding as in the proof of \cite[Lemma 3.2]{Le} (see \cite[Lemma 3.6]{DGKLWY}) we can see that, for almost all $x\in B_0$,
$$
\|T(f\mathcal{X}_{3B_0})(x)\|_E\leq C\|T\|_{L_E^1(\Omega ,\mu)\rightarrow L_E^{1,\infty}(\Omega ,\mu)}|f(x)|+\mathcal{M}_{T,B_0}(f)(x).
$$
This property is used in the proof of Theorem \ref{Th1.2} for each one of the operators we consider.
\subsection{Proof of Theorem \ref{Th1.2} for $\phi_{*,b}^{\lambda ,m}$}\label{S3.1}
We are going to see the properties that we need in order that the arguments developed in the proof of \cite[Theorems 1.3 and 3.7]{DGKLWY} work for $\phi_{*,b}^{\lambda ,m}$.

By proceeding as in Section \ref{S2.1} we can see that
$$
\phi _{*,b}^{\lambda ,1}(f)(x)\leq \mathcal{M}_\lambda (|b(x)-b(\cdot )|f)(x),\quad x\in (0,\infty ).
$$
Then, Theorem \ref{Th1.2} for $\phi _{*,b}^{\lambda ,1}$ can be deduced from \cite[Theorem 1.2]{GVW}. In \cite{GHST} $L^p$-weighted inequalities were obtained for the maximal commutator operator $\mathcal{M}_b^m$ defined by
$$
\mathcal{M}_b^m(f)(x)=\mathcal{M}(|b(x)-b(\cdot)|^mf)(x),\quad x\in (0,\infty ),
$$
where $\mathcal{M}$ is the euclidean Hardy-Littlewood operator. In \cite{GHST} they did not study how the $L^p(w)$-norm of the operators depend on the weight $w$.

\begin{prop}
Let $\lambda >0$. Assume that $\phi$ is a differentiable function on $(0,\infty )$ such that $|\phi ^{(k)}(x)|\leq C(1+x)^{-2\lambda -2-k}$, $x\in (0,\infty )$ and $k=0,1$. Let $\ell \in \mathbb{N}$. We consider the operator
$$
[T_\ell (f)(x)](t)=(\phi _t\# _\lambda f)(x),\quad t\in \Big[\frac{1}{\ell},\ell\Big]\mbox{ and }x\in (0,\infty ).
$$
We have that

(i) For every $f\in C_c^\infty (0,\infty )$, the space of smooth functions with compact support on $(0,\infty )$,
$$
T_\ell (f)(x)=\int_0^\infty \,_\lambda \tau _x(\phi_\centerdot )(y)f(y)y^{2\lambda }dy,\quad x\not\in {\rm supp }f,
$$
where the integral is understood in the Banach space $(C([1/\ell ,\ell]),\|\cdot \|_\ell)$ of the continuous functions on $[1/\ell ,\ell]$ where
$$
\|g\|_\ell =\max_{t\in [1/\ell ,\ell]}|g(t)|,\quad g\in C\Big(\Big[\frac{1}{\ell},\ell\Big]\Big).
$$

(ii) We define
$$
K_t(x,y)=\,_\lambda \tau _x(\phi _t)(y),\quad t,x,y\in (0,\infty ).
$$
Then, there exists $C>0$ such that
\begin{equation}\label{3.1}
    \|K_{\centerdot} (x,y)\|_\infty \leq \frac{C}{m_\lambda (B(x,|x-y|))},\quad x,y\in (0,\infty ),\;x\not=y,
\end{equation}
and
\begin{equation}\label{3.2}
\|\partial _xK_\centerdot (x,y)\|_\infty +\|\partial _yK_\centerdot (x,y\|_\infty \leq \frac{C}{|x-y|m_\lambda (B(x,|x-y|))},\quad x,y\in (0,\infty ),\;x\not=y.
\end{equation}
\end{prop}

\begin{proof}
Firstly we prove \eqref{3.1}. Since $|\phi (x)|\leq C(1+x)^{-2\lambda -2}$, $x\in (0,\infty )$, as in \cite[(2.4)]{YY} we get
$$
|K_t(x,y)|\leq \frac{C}{m_\lambda (B(x,t))+m_\lambda (B(x,|x-y|))}\frac{t}{t+|x-y|},\quad t,x,y\in (0,\infty).
$$
Then,
$$
\|K_\centerdot(x,y)\|_\infty \leq \frac{C}{m_\lambda (B(x,|x-y|))},\quad x,y\in (0,\infty ),x\not=y.
$$

We have, for every $t,x,y\in (0,\infty )$,
$$
\partial _xK_t(x,y)=\frac{\Gamma (\lambda +1/2)}{\Gamma (\lambda )\sqrt{\pi}t^{2\lambda +2}}\int_0^\pi \phi'\left(\frac{\sqrt{x^2+y^2-2xy\cos \theta}}{t}\right)\frac{x-y\cos \theta}{\sqrt{x^2+y^2-2xy\cos \theta}}\sin ^{2\lambda -1}\theta d\theta.
$$
Since $|\phi'(x)|\leq C(1+x)^{-2\lambda -3}$, $x\in (0,\infty )$, and $|x-y\cos \theta |\leq \sqrt{x^2+y^2-2xy\cos \theta}$, $x,y\in (0,\infty )$ and $\theta \in (0,\pi)$, we get
$$
|\partial _xK_t(x,y)|\leq Ct\int_0^\pi \frac{\sin ^{2\lambda -1}\theta}{(t^2+x^2+y^2-2xy\cos \theta )^{\lambda +3/2}}d\theta,\quad t,x,y\in (0,\infty ).
$$
By proceeding as in \cite[(2.8)]{YY} we obtain
\begin{align*}
|\partial _xK_t(x,y)|&\leq C\frac{1}{m_\lambda (B(x,t))+m_\lambda (B(x,|x-y|))}\frac{t}{(t+|x-y|)^2}\\
&\leq \frac{C}{|x-y|m_\lambda (B(x,|x-y|))},\quad t,x,y\in (0,\infty),\;x\not=y.
\end{align*}
Then,
$$
\|\partial _xK_\centerdot (x,y)\|_\infty\leq \frac{C}{|x-y|m_\lambda (B(x,|x-y|))},\quad x,y\in (0,\infty ),\;x\not=y.
$$
By simmetry, \eqref{3.2} is established.

Let $f\in C_c^\infty (0,\infty )$ and $\ell \in \mathbb{N}$. According to \eqref{3.1} the integral defining $T_\ell (t)(x)$ is $\|\cdot \|_\ell$-Bochner convergent for every $x\not\in {\rm supp }f$. The dual space of $C([1/\ell ,\ell])$ is the space $\mathbb{M}([1/\ell ,\ell ])$ of complex measures in $[1/\ell ,\ell]$. Let $\mu \in \mathbb{M}([1/\ell ,\ell ])$. We have that
$$
    \int_0^\infty \int_{1/\ell}^\ell K_t(x,y)d\mu (t)f(y)y^{2\lambda}dy=\int_{1/\ell}^\ell (\phi _t\#_\lambda f)(x)d\mu (t),\quad x\not\in {\rm supp }f.
$$
This equality is justified because, from \eqref{3.1} we get
$$
\int_0^\infty \int_{1/\ell}^\ell |K_t(x,y)|d|\mu|(t)|f(y)|y^{2\lambda }dy\leq C|\mu|\Big(\Big[\frac{1}{\ell },\ell\Big]\Big)\int_{{\rm supp }f}\frac{dy}{m_\lambda (B(x,|x-y|))}<\infty,\quad x\not\in {\rm supp }f.
$$
Then, for every $x\not\in {\rm supp }f$
$$
\left(\int_0^\infty \,_\lambda \tau _x(\phi_\centerdot )(y)f(y)y^{2\lambda }dy\right)(t)=(\phi _t\# _\lambda f)(x),\quad \mbox{ a.e. }t\in \Big(\frac{1}{\ell},\ell\Big).
$$
\end{proof}

It was established in Section \ref{S2.1} that $\phi_*^\lambda(f)\leq C\mathcal{M}_\lambda (f)$. Then, $\phi _*^\lambda$ is bounded from $L^1((0,\infty ),m_\lambda)$ into $L^{1,\infty }((0,\infty ),m_\lambda )$.

We now define the operator
$$
\mathbb{M}_\phi ^{\lambda ,\#}(f)(x)=\sup_{\substack{Q\ni x\\Q\mbox{ interval}}}\esssup_{y_1,y_2\in Q}\Big|\phi _*^\lambda (f\mathcal{X}_{(0,\infty )\setminus (3Q)})(y_1)-\phi _*^\lambda (f\mathcal{X}_{(0,\infty )\setminus (3Q)})(y_2)\Big|,\quad x\in (0,\infty ).
$$
According to \eqref{3.2} by proceeding as in the proof of Proposition \ref{propo 2.2} we can see the following.
\begin{prop} \label{Prop 3.2}
Let $\lambda >0$. Suppose that $\phi$ is a differentiable function on $(0,\infty )$ such that $|\phi ^{(k)}(x)|\leq C(1+x)^{-2\lambda -1-k}$, $x\in (0,\infty )$ and $k=0,1$. Then, the operator $\mathbb{M}_\phi ^{\lambda ,\#}$ is bounded from $L^p((0,\infty ),m_\lambda )$ into itself, for every $1<p<\infty$, and from $L^1((0,\infty ),m_\lambda )$ into $L^{1,\infty }((0,\infty ),m_\lambda )$.
\end{prop}

As it was commented at the beginning of Section \ref{S2} from Proposition \ref{Prop 3.2} we can deduce that the operator $\mathcal{M}_\phi ^\lambda$ defined by
$$
\mathcal{M}_\phi ^\lambda (f)(x)=\sup_{\substack{Q\ni x\\Q\mbox{ interval}}}\|\phi _*^\lambda (f\mathcal{X}_{(0,\infty )\setminus (3Q)})\|_{L^\infty (Q)},\quad x\in (0,\infty ),
$$
is bounded from $L^1((0,\infty ),m_\lambda )$ into $L^{1,\infty}((0,\infty ),m_\lambda ))$.

We have established all the properties that we need to prove Theorem \ref{Th1.2} by using the arguments in the proofs of \cite[Theorem 1.3 and 3.7]{DGKLWY}.

\subsection{Proof of Theorem \ref{Th1.2} for $g_{\phi ,b}^{\lambda, m}$}
By using the properties that were seen in Section \ref{S2.2} we can prove Theorem \ref{Th1.2} for $g_{\phi ,b}^{\lambda, m}$ by proceeding as in \cite[Theorems 1.3 and 3.7]{DGKLWY}.

\subsection{Proof of Theorem \ref{Th1.2} for $V_{\rho ,b}^{\lambda ,m}(\{\phi _t\}_{t>0})$}
By taking into account the properties established in Section \ref{S2.3}, in order to prove Theorem \ref{Th1.2} for $V_{\rho,b} ^{\lambda,m} (\{\phi _t\}_{t>0})$ (see also Section \ref{S3.1}) it is sufficient to show that
$$
\|\{\,_\lambda \tau _x(\phi _t)(y)\}_{t>0}\|_{v_\rho }\leq \frac{C}{m_\lambda(B(x,|x-y|))},\quad x,y\in (0,\infty ),\;x\not=y.
$$

We have that
$$
\|\{\,_\lambda \tau _x(\phi _t)(y)\}_{t>0}\|_{v_\rho }\leq \int_0^\infty |\partial _t\,_\lambda\tau _x(\phi _t)(y)|dt,\quad x,y\in (0,\infty ).
$$
Since $\partial _t\,_\lambda\tau _x(\phi _t)(y)=-\frac{1}{t}\,_\lambda \tau_y(\Phi _t)(x)$, $t,x,y\in (0,\infty )$, where $\Phi (x)=(2\lambda +1)\phi (x)+x\phi '(x)$, $x\in (0,\infty )$, we deduce that
\begin{align*}
    \|\{\,_\lambda \tau _x(\phi _t)(y)\}_{t>0}\|{v_\rho} &\leq C\int_0^\infty \Big|\,_\lambda \tau_y(\Phi _t)(x)\Big|\frac{dt}{t}\\
    &\leq C\int_0^\pi \sin ^{2\lambda -1}\theta \int_0^\infty \Big|\Phi \Big(\frac{\sqrt{x^2+y^2-2xy\cos \theta }}{t}\Big)\Big|\frac{dt}{t^{2\lambda +2}}d\theta\\
    &\leq C\int_0^\pi \frac{\sin ^{2\lambda -1}\theta}{(x^2+y^2-2xy\cos \theta)^{\lambda +1/2}}\int_0^\infty u^{2\lambda }|\Phi (u)|dud\theta .
\end{align*}
Since $|\Phi (u)|\leq C(1+u)^{-2\lambda -2}$, $u\in (0,\infty )$, according to \cite[Lemma 3.1]{BCN} it follows that
$$
\|\{\,_\lambda \tau _x(\phi _t)(y)\}_{t>0}\|_{v_\rho }\leq \frac{C}{m_\lambda (B(x,|x-y|))},\quad x,y\in (0,\infty ),\;x\not=y.
$$

\subsection{Applications of Theorem \ref{Th1.2}}
The special cases of Theorem \ref{Th1.1} presented in Section \ref{S2.4} also can be seen as special cases of Theorem \ref{Th1.2}.


\begin{thebibliography}{10}

\bibitem{CPPV}
{\sc Alonso, J. M.~C., Plinio, F.~D., Parissis, I., and Vempati, M.~N.}
\newblock A metric approach to sparse domination, 2020.

\bibitem{Al}
{\sc Altenburg, G.}
\newblock Bessel-{T}ransformationen in {R}\"{a}umen von {G}rundfunktionen
  \"{u}ber dem {I}ntervall {$\Omega =(0,\,\infty )$} und deren
  {D}ualr\"{a}umen.
\newblock {\em Math. Nachr. 108\/} (1982), 197--218.

\bibitem{AK}
{\sc Andersen, K.~F., and Kerman, R.~A.}
\newblock Weighted norm inequalities for generalized {H}ankel conjugate
  transformations.
\newblock {\em Studia Math. 71}, 1 (1981/82), 15--26.

\bibitem{BCN}
{\sc Betancor, J.~J., Castro, A.~J., and Nowak, A.}
\newblock Calder\'{o}n--{Z}ygmund operators in the {B}essel setting.
\newblock {\em Monatsh. Math. 167}, 3-4 (2012), 375--403.

\bibitem{BDFS}
{\sc Betancor, J.~J., Dalmasso, E., Fari\~{n}a, J.~C., and Scotto, R.}
\newblock Bellman functions and dimension free {$L^p$}-estimates for the
  {R}iesz transforms in {B}essel settings.
\newblock {\em Nonlinear Anal. 197\/} (2020), 111850, 24.

\bibitem{BDT}
{\sc Betancor, J.~J., Dziuba\'{n}ski, J., and Torrea, J.~L.}
\newblock On {H}ardy spaces associated with {B}essel operators.
\newblock {\em J. Anal. Math. 107\/} (2009), 195--219.

\bibitem{BFS}
{\sc Betancor, J.~J., Fari\~{n}a, J.~C., and Sanabria, A.}
\newblock On {L}ittlewood-{P}aley functions associated with {B}essel operators.
\newblock {\em Glasg. Math. J. 51}, 1 (2009), 55--70.

\bibitem{BHNV}
{\sc Betancor, J.~J., Harboure, E., Nowak, A., and Viviani, B.}
\newblock Mapping properties of fundamental operators in harmonic analysis
  related to {B}essel operators.
\newblock {\em Studia Math. 197}, 2 (2010), 101--140.

\bibitem{BC}
{\sc Brandolini, L., and Colzani, L.}
\newblock Bochner-{R}iesz means with negative index of radial functions in
  {S}obolev spaces.
\newblock {\em Rend. Circ. Mat. Palermo (2) 42}, 1 (1993), 117--128.

\bibitem{BD}
{\sc Bui, T.~A., and Duong, X.~T.}
\newblock Sharp weighted estimates for square functions associated to operators
  on spaces of homogeneous type.
\newblock {\em J. Geom. Anal. 30}, 1 (2020), 874--900.

\bibitem{BDL}
{\sc Bui, T.~A., Duong, X.~T., and Ly, F.~K.}
\newblock Maximal function characterizations for new local {H}ardy-type spaces
  on spaces of homogeneous type.
\newblock {\em Trans. Amer. Math. Soc. 370}, 10 (2018), 7229--7292.

\bibitem{CJRW1}
{\sc Campbell, J.~T., Jones, R.~L., Reinhold, K., and Wierdl, M.}
\newblock Oscillation and variation for the {H}ilbert transform.
\newblock {\em Duke Math. J. 105}, 1 (2000), 59--83.

\bibitem{CJRW2}
{\sc Campbell, J.~T., Jones, R.~L., Reinhold, K., and Wierdl, M.}
\newblock Oscillation and variation for singular integrals in higher
  dimensions.
\newblock {\em Trans. Amer. Math. Soc. 355}, 5 (2003), 2115--2137.

\bibitem{CSZ}
{\sc Cao, M., Si, Z., and Zhang, J.}
\newblock Weak and strong types estimates for square functions associated with
  operators, 2020.

\bibitem{CD}
{\sc Carro, M.~J., and Domingo-Salazar, C.}
\newblock Stein's square function {$G_\alpha$} and sparse operators.
\newblock {\em J. Geom. Anal. 27}, 2 (2017), 1624--1635.

\bibitem{CDY}
{\sc Chen, P., Duong, X.~T., and Yan, L.}
\newblock {$L^p$}-bounds for {S}tein's square functions associated to operators
  and applications to spectral multipliers.
\newblock {\em J. Math. Soc. Japan 65}, 2 (2013), 389--409.

\bibitem{ChD}
{\sc Chen, Y., and Ding, Y.}
\newblock Commutators of {L}ittlewood-{P}aley operators.
\newblock {\em Sci. China Ser. A 52}, 11 (2009), 2493--2505.

\bibitem{Chr}
{\sc Christ, M.}
\newblock A {$T(b)$} theorem with remarks on analytic capacity and the {C}auchy
  integral.
\newblock {\em Colloq. Math. 60/61}, 2 (1990), 601--628.

\bibitem{CW}
{\sc Coifman, R.~R., and Weiss, G.}
\newblock {\em Analyse harmonique non-commutative sur certains espaces
  homog\`enes}.
\newblock Lecture Notes in Mathematics, Vol. 242. Springer-Verlag, Berlin-New
  York, 1971.
\newblock \'{E}tude de certaines int\'{e}grales singuli\`eres.

\bibitem{CTV}
{\sc Colzani, L., Travaglini, G., and Vignati, M.}
\newblock Bochner-{R}iesz means of functions in weak-{$L^p$}.
\newblock {\em Monatsh. Math. 115}, 1-2 (1993), 35--45.

\bibitem{CZ}
{\sc Cui, X., and Zhang, J.}
\newblock Oscillation and variation for {R}iesz transform in setting of
  {B}essel operators on {$H^1$} and {BMO}.
\newblock {\em Front. Math. China 15}, 4 (2020), 617--647.

\bibitem{DX}
{\sc Ding, Y., and Xue, Q.}
\newblock Endpoint estimates for commutators of a class of {L}ittlewood-{P}aley
  operators.
\newblock {\em Hokkaido Math. J. 36}, 2 (2007), 245--282.

\bibitem{DGKLWY}
{\sc Duong, X.~T., Gong, R., Kuffner, M.-J.~S., Li, J., Wick, B.~D., and Yang,
  D.}
\newblock Two weight commutators on spaces of homogeneous type and
  applications.
\newblock {\em J. Geom. Anal. 31}, 1 (2021), 980--1038.

\bibitem{DLMWY}
{\sc Duong, X.~T., Li, J., Mao, S., Wu, H., and Yang, D.}
\newblock Compactness of {R}iesz transform commutator associated with {B}essel
  operators.
\newblock {\em J. Anal. Math. 135}, 2 (2018), 639--673.

\bibitem{DLWY}
{\sc Duong, X.~T., Li, J., Wick, B.~D., and Yang, D.}
\newblock Factorization for {H}ardy spaces and characterization for {BMO}
  spaces via commutators in the {B}essel setting.
\newblock {\em Indiana Univ. Math. J. 66}, 4 (2017), 1081--1106.

\bibitem{D}
{\sc Dziuba\'{n}ski, J.}
\newblock Hardy spaces associated with semigroups generated by {B}essel
  operators with potentials.
\newblock {\em Houston J. Math. 34}, 1 (2008), 205--234.

\bibitem{DPW}
{\sc Dziuba\'{n}ski, J., Preisner, M., and Wr\'{o}bel, B.~a.}
\newblock Multivariate {H}\"{o}rmander-type multiplier theorem for the {H}ankel
  transform.
\newblock {\em J. Fourier Anal. Appl. 19}, 2 (2013), 417--437.

\bibitem{EMOT}
{\sc Erd\'{e}lyi, A., Magnus, W., Oberhettinger, F., and Tricomi, F.~G.}
\newblock {\em Tables of integral transforms. {V}ol. {II}}.
\newblock McGraw-Hill Book Company, Inc., New York-Toronto-London, 1954.
\newblock Based, in part, on notes left by Harry Bateman.

\bibitem{FPW}
{\sc Fu, Z., Pozzi, E., and Wu, Q.}
\newblock Commutators of maximal functions on spaces of homogeneous type and
  their weighted, local versions, 2019.

\bibitem{GHST}
{\sc Garc\'{\i}a-Cuerva, J., Harboure, E., Segovia, C., and Torrea, J.~L.}
\newblock Weighted norm inequalities for commutators of strongly singular
  integrals.
\newblock {\em Indiana Univ. Math. J. 40}, 4 (1991), 1397--1420.

\bibitem{GRS}
{\sc Garg, R., Roncal, L., and Shrivastava, S.}
\newblock Quantitative weighted estimates for {R}ubio de {F}rancia's
  {L}ittlewood-{P}aley square function.
\newblock {\em J. Geom. Anal. 31}, 1 (2021), 748--771.

\bibitem{GS}
{\sc Garrig\'{o}s, G., and Seeger, A.}
\newblock Characterizations of {H}ankel multipliers.
\newblock {\em Math. Ann. 342}, 1 (2008), 31--68.

\bibitem{GVW}
{\sc Gong, R., Vempati, M.~N., and Wu, Q.}
\newblock A note on two weight commutators of maximal functions on spaces of
  homogeneous type, 2020.

\bibitem{Gra}
{\sc Grafakos, L.}
\newblock {\em Classical {F}ourier analysis}, second~ed., vol.~249 of {\em
  Graduate Texts in Mathematics}.
\newblock Springer, New York, 2008.

\bibitem{Ha}
{\sc Haimo, D.~T.}
\newblock Integral equations associated with {H}ankel convolutions.
\newblock {\em Trans. Amer. Math. Soc. 116\/} (1965), 330--375.

\bibitem{Hi}
{\sc Hirschman, Jr., I.~I.}
\newblock Variation diminishing {H}ankel transforms.
\newblock {\em J. Analyse Math. 8\/} (1960/61), 307--336.

\bibitem{HK}
{\sc Hyt\"{o}nen, T., and Kairema, A.}
\newblock Systems of dyadic cubes in a doubling metric space.
\newblock {\em Colloq. Math. 126}, 1 (2012), 1--33.

\bibitem{Hy}
{\sc Hyt\"{o}nen, T.~P.}
\newblock The sharp weighted bound for general {C}alder\'{o}n-{Z}ygmund
  operators.
\newblock {\em Ann. of Math. (2) 175}, 3 (2012), 1473--1506.

\bibitem{JKRW}
{\sc Jones, R.~L., Kaufman, R., Rosenblatt, J.~M., and Wierdl, M.}
\newblock Oscillation in ergodic theory.
\newblock {\em Ergodic Theory Dynam. Systems 18}, 4 (1998), 889--935.

\bibitem{JSW}
{\sc Jones, R.~L., Seeger, A., and Wright, J.}
\newblock Strong variational and jump inequalities in harmonic analysis.
\newblock {\em Trans. Amer. Math. Soc. 360}, 12 (2008), 6711--6742.

\bibitem{KP}
{\sc Kania, E., and Preisner, M.}
\newblock Sharp multiplier theorem for multidimensional {B}essel operators.
\newblock {\em J. Fourier Anal. Appl. 25}, 5 (2019), 2419--2446.

\bibitem{K}
{\sc Kim, J.}
\newblock Endpoint bounds of square functions associated with {H}ankel
  multipliers.
\newblock {\em Studia Math. 228}, 2 (2015), 123--151.

\bibitem{LMR}
{\sc Lacey, M.~T., Mena, D., and Reguera, M.~C.}
\newblock Sparse bounds for {B}ochner-{R}iesz multipliers.
\newblock {\em J. Fourier Anal. Appl. 25}, 2 (2019), 523--537.

\bibitem{LN}
{\sc Langowski, B., and Nowak, A.}
\newblock Mapping properties of fundamental harmonic analysis operators in the
  exotic {B}essel framework.
\newblock {\em J. Math. Anal. Appl. 499}, 2 (2021), 125061, 36.

\bibitem{LeMX1}
{\sc Le~Merdy, C., and Xu, Q.}
\newblock Maximal theorems and square functions for analytic operators on
  {$L^p$}-spaces.
\newblock {\em J. Lond. Math. Soc. (2) 86}, 2 (2012), 343--365.

\bibitem{LeMX2}
{\sc Le~Merdy, C., and Xu, Q.}
\newblock Strong {$q$}-variation inequalities for analytic semigroups.
\newblock {\em Ann. Inst. Fourier (Grenoble) 62}, 6 (2012), 2069--2097 (2013).

\bibitem{Leb}
{\sc Lebedev, N.~N.}
\newblock {\em Special functions and their applications}.
\newblock Dover Publications, Inc., New York, 1972.
\newblock Revised edition, translated from the Russian and edited by Richard A.
  Silverman, Unabridged and corrected republication.

\bibitem{Le1}
{\sc Lerner, A.~K.}
\newblock A simple proof of the {$A_2$} conjecture.
\newblock {\em Int. Math. Res. Not. IMRN}, 14 (2013), 3159--3170.

\bibitem{Le}
{\sc Lerner, A.~K.}
\newblock On pointwise estimates involving sparse operators.
\newblock {\em New York J. Math. 22\/} (2016), 341--349.

\bibitem{LeNa}
{\sc Lerner, A.~K., and Nazarov, F.}
\newblock Intuitive dyadic calculus: the basics.
\newblock {\em Expo. Math. 37}, 3 (2019), 225--265.

\bibitem{LO}
{\sc Lerner, A.~K., and Ombrosi, S.}
\newblock Some remarks on the pointwise sparse domination.
\newblock {\em J. Geom. Anal. 30}, 1 (2020), 1011--1027.

\bibitem{LOR}
{\sc Lerner, A.~K., Ombrosi, S., and Rivera-R\'{\i}os, I.~P.}
\newblock On pointwise and weighted estimates for commutators of
  {C}alder\'{o}n-{Z}ygmund operators.
\newblock {\em Adv. Math. 319\/} (2017), 153--181.

\bibitem{Li}
{\sc Li, K.}
\newblock Sparse domination theorem for multilinear singular integral operators
  with {$L^r$}-{H}\"{o}rmander condition.
\newblock {\em Michigan Math. J. 67}, 2 (2018), 253--265.

\bibitem{LPRR}
{\sc Li, K., P\'{e}rez, C., Rivera-R\'{\i}os, I.~P., and Roncal, L.}
\newblock Weighted norm inequalities for rough singular integral operators.
\newblock {\em J. Geom. Anal. 29}, 3 (2019), 2526--2564.

\bibitem{La}
{\sc Liu, L.}
\newblock Weighted weak type estimates for commutators of {L}ittlewood-{P}aley
  operator.
\newblock {\em Japan. J. Math. (N.S.) 29}, 1 (2003), 1--13.

\bibitem{Lori}
{\sc Lorist, E.}
\newblock On pointwise {$\ell^r$}-sparse domination in a space of homogeneous
  type.
\newblock {\em J. Geom. Anal.\/} (2020).

\bibitem{MTX}
{\sc Ma, T., Torrea, J.~L., and Xu, Q.}
\newblock Weighted variation inequalities for differential operators and
  singular integrals.
\newblock {\em J. Funct. Anal. 268}, 2 (2015), 376--416.

\bibitem{MS}
{\sc Muckenhoupt, B., and Stein, E.~M.}
\newblock Classical expansions and their relation to conjugate harmonic
  functions.
\newblock {\em Trans. Amer. Math. Soc. 118\/} (1965), 17--92.

\bibitem{NTV}
{\sc Nazarov, F., Treil, S., and Volberg, A.}
\newblock Weak type estimates and {C}otlar inequalities for
  {C}alder\'{o}n-{Z}ygmund operators on nonhomogeneous spaces.
\newblock {\em Internat. Math. Res. Notices}, 9 (1998), 463--487.

\bibitem{NS}
{\sc Nowak, A., and Stempak, K.}
\newblock Potential operators associated with {H}ankel and {H}ankel-{D}unkl
  transforms.
\newblock {\em J. Anal. Math. 131\/} (2017), 277--321.

\bibitem{Pe}
{\sc Pereyra, M.~C.}
\newblock Dyadic harmonic analysis and weighted inequalities: the sparse
  revolution.
\newblock In {\em New trends in Applied Harmonic Analysis}, vol.~2 of {\em
  Applied and numeraical harmonic analysis}. Birk\"auser, 2019, pp.~159--239.

\bibitem{RRT}
{\sc Rubio~de Francia, J.~L., Ruiz, F.~J., and Torrea, J.~L.}
\newblock Calder\'{o}n-{Z}ygmund theory for operator-valued kernels.
\newblock {\em Adv. in Math. 62}, 1 (1986), 7--48.

\bibitem{Stein1}
{\sc Stein, E.~M.}
\newblock Localization and summability of multiple {F}ourier series.
\newblock {\em Acta Math. 100\/} (1958), 93--147.

\bibitem{Stein}
{\sc Stein, E.~M.}
\newblock {\em Topics in harmonic analysis related to the {L}ittlewood-{P}aley
  theory}.
\newblock Annals of Mathematics Studies, No. 63. Princeton University Press,
  Princeton, N.J.; University of Tokyo Press, Tokyo, 1970.

\bibitem{Stem}
{\sc Stempak, K.}
\newblock La th\'{e}orie de {L}ittlewood-{P}aley pour la transformation de
  {F}ourier-{B}essel.
\newblock {\em C. R. Acad. Sci. Paris S\'{e}r. I Math. 303}, 1 (1986), 15--18.

\bibitem{Vi}
{\sc Villani, M.}
\newblock Riesz transforms associated to {B}essel operators.
\newblock {\em Illinois J. Math. 52}, 1 (2008), 77--89.

\bibitem{WYZ}
{\sc Wu, H., Yang, D., and Zhang, J.}
\newblock Oscillation and variation for semigroups associated with {B}essel
  operators.
\newblock {\em J. Math. Anal. Appl. 443}, 2 (2016), 848--867.

\bibitem{YY}
{\sc Yang, D., and Yang, D.}
\newblock Real-variable characterizations of {H}ardy spaces associated with
  {B}essel operators.
\newblock {\em Anal. Appl. (Singap.) 9}, 3 (2011), 345--368.

\end{thebibliography}

\end{document}